\documentclass[12pt]{hacked} 
\usepackage[utf8]{inputenc}
\pdfoutput=1
\title[From Nesterov's Estimate Sequence to Riemannian Acceleration]{From Nesterov's Estimate Sequence to Riemannian Acceleration}

\usepackage[stable]{footmisc}
\usepackage{times}
\usepackage{pgfplots}
\usepackage{wrapfig}
\usepackage{caption}
\usepackage{mdframed}
\usepackage{nicefrac}
\usepackage{microtype}
\makeatletter
\renewcommand{\theequation}{\thesection.\arabic{equation}}
\renewcommand{\thesection}{\arabic{section}}
\numberwithin{equation}{section}

\newcommand{\eqnum}{\leavevmode\hfill\refstepcounter{equation}\textup{\tagform@{\theequation}}}
 
\newcommand{\algorithmstyle}[1]{\renewcommand{\algocf@style}{#1}}
\makeatother

\usepackage{thmtools} 
\declaretheorem[ shaded={rulecolor=black, rulewidth=0.5pt, bgcolor=gray!7}, name=Theorem, numberwithin=section]{thmbox}
\declaretheorem[ shaded={rulecolor=black, rulewidth=0.5pt, bgcolor=white}, name=Lemma ,numberwithin=section]{lembox}
\declaretheorem[ name=Corollary,numberwithin=section]{corobox}

\declaretheorem[ shaded={rulecolor=black, rulewidth=0.5pt, bgcolor=gray!7}, name=Algorithm ]{alg}

\newcommand{\inp}[2]{\left\langle#1, #2 \right\rangle}
\newcommand{\inpp}[2]{\langle#1, #2 \rangle}
\newcommand{\norm}[1]{\left\lVert#1\right\rVert}
\newcommand{\normp}[1]{ \lVert#1\rVert}
\newcommand{\onorm}[1]{\left\lVert#1\right\rVert_{\textrm{op}}}

\newcommand{\OO}[1]{O\left( #1\right)}
\newcommand{\re}{\mathbb{R}}

\newcommand{\eps}{\epsilon}

\newcommand{\exm}[2]{\mathrm{Exp}_{#1}\left(#2 \right)  }
\newcommand{\expm}{\mathrm{Exp} }
\newcommand{\lm}[2]{\mathrm{Exp}^{-1}_{#1}\left(#2 \right)  }

\newcommand{\grad}{\mathsf{Grad}}
\newcommand{\mirr}{\mathsf{Mirr}}

\newcommand{\DD}[3]{d_{#1}(#2,#3)}
\newcommand{\dd}[2]{d\left(#1,#2\right)}
\newcommand{\T}{T}

\newcommand{\drr}[1]{\delta_{#1}}
\newcommand{\tri}[1]{\zeta_{#1}}
\newcommand{\cc}{\xi}
\newcommand{\ddr}[1]{\delta_{#1}}
\newcommand{\init}{D_{0}}
\newcommand{\xg}{x}
\newcommand{\xm}{w}
\newcommand{\wg}{ \alpha }
\newcommand{\wm}{ \beta }
\newcommand{\sg}{ \gamma}
\newcommand{\sm}{\eta}
\newcommand{\drop}{\Delta_{\sg}}
\newcommand{\err}{\mathcal{E}}

\newcommand{\vd}{T_{\kappa}}

\newcommand{\WW}{\widetilde{W}}
\newcommand{\XX}{\widetilde{X}}
\newcommand{\NN}{\widetilde{\nabla}}
\newcommand{\CC}{\widetilde{C}}
\newcommand{\cho}{\phi}
\newcommand{\cht}{\psi}
\newcommand{\ctt}{\tau}
\newcommand{\cde}{\theta}
\newcommand{\co}{\mathcal{C}_{\mu,L,\sg}}
\newcommand{\ct}{\mathcal{D}_{\kappa,\mu,\sg}}
\newcommand{\rr}{\lambda}
\newcommand{\upp}{\sigma}
 \author{\Name{Kwangjun Ahn} \Email{kjahn@mit.edu} 
 \AND
 \Name{Suvrit Sra} \Email{suvrit@mit.edu}\\
 \addr Department of Electrical Engineering and Computer Science, Massachusetts Institute of Technology}

\begin{document}

\maketitle

\vspace*{-.8cm}
\begin{abstract}%
We propose the first global accelerated gradient method for Riemannian manifolds. Toward establishing our result we revisit Nesterov's estimate sequence technique and develop an alternative analysis for it that may also be of independent interest. Then, we extend this analysis to the Riemannian setting, localizing the key difficulty due to non-Euclidean structure into a certain ``metric distortion.'' We control this distortion by developing a novel geometric inequality, which permits us to propose and analyze a Riemannian counterpart to Nesterov's accelerated gradient method. 
\end{abstract}

\section{Introduction}
\vspace*{-3pt}
First-order methods enjoy a well-developed theory for convex problems, while also demonstrating practical success in current machine learning tasks involving \emph{non-convex} problems. But non-convex problems are in general intractable, so the corresponding theory of first-order methods either limits itself to local (stationarity) results, or relies on special structure that permits sharper global analysis.

A promising instance of such special structure arises via non-Euclidean geometry. Specifically, via the concept of \emph{geodesic convexity} (\emph{g-convexity}), which defines convexity along geodesics in metric spaces~\citep{gromov1978manifolds,burago2001course,bridson2013metric}. 
Problems that are non-convex under a Euclidean view can sometimes be transformed into g-convex problems in a non-Euclidean space, potentially uncovering their \emph{tractability}. This viewpoint has proved fruitful in several applications (e.g. see \citep[Section 1.1]{zhang2016first}), as well as towards tackling some theoretical questions in computer science, mathematics, and physics---see \citep{burgisser2019towards,goyal2019sampling} and references therein, as well as \S\ref{related} of this paper.

Paralleling this viewpoint, several works have sought to analyze first-order methods in non-Euclidean settings, primarily in Riemannian manifolds~\citep{udriste1994convex, absil2009optimization} and CAT(0) spaces~\citep{bacak2014convex}. Earlier studies focus on \emph{asymptotic} analysis, while \citet{zhang2016first} obtain the first \emph{non-asymptotic}  global iteration complexity analysis for Riemannian (stochastic) gradient descent methods, assuming g-convexity. Subsequently, iteration-complexity results were established for Riemannian proximal-point methods~\citep{bento2017iteration}, Frank-Wolfe schemes~\citep{weber2019nonconvex}, variance reduced methods~\citep{zhang2016riemannian,kasai2016riemannian,zhang2018r,zhou2019faster}, trust-region methods \citep{agarwal2018adaptive}, among others.
  
Despite this progress, a landmark achievement of Euclidean convex optimization remains elusive in the Riemannian setting: that is, obtaining an analogue of Nesterov's accelerated gradient method~\citep{nesterov1983method}. This gap motivates the central question of our paper:
\begin{center}
  \emph{Is it possible to develop accelerated gradient methods for Riemannian manifolds?}
\end{center}%
This natural question, however, is believed to be highly non-trivial, especially because Nesterov's analysis heavily relies on the linear structure of Euclidean space. In particular, recent efforts were able to answer this question only \emph{partially}. See \S\ref{related} for details.

\subsection{Overview of our main results}
\vspace*{-4pt}   
In this paper, we take a major step toward answering this question by developing the \emph{first global accelerated first-order method} for Riemannian manifolds (formal statement in Theorem~\ref{thm:main}):
\begin{thmbox}[Informal] \label{thm:main:informal}
  Let $f$ be $L$-smooth and $\mu$-strongly convex in a geodesic sense. Then, there exists a \emph{computationally tractable} optimization algorithm satisfying\\[-10pt]
  \begin{equation*}
    f(x_t)-f(x_*)=\OO{(1-\cc_1)(1-\cc_2)\cdots (1-\cc_t)},
    \vspace*{-3pt}
  \end{equation*}
  where $\{\cc_t\}$ satisfies (i) $\{\cc_{t}\}_{t\ge 1} > \nicefrac{\mu}{L}$ ({\bf strictly faster} than gradient descent); and (ii) $\exists \rr\in(0,1)$ such that $\forall\eps>0$, $|\cc_t-\sqrt{2\mu\drop}|\leq \eps$, for $t\geq \Omega\bigl(\frac{\log(1/\eps)}{\log(1/\rr)}\bigr)$  (eventually achieves {\bf full acceleration}).
\end{thmbox}

To establish Theorem~\ref{thm:main:informal}, we revisit and develop an alternative analysis to Nesterov's  analysis for the Euclidean case (\S\ref{sec:euclid}).
We use the technique of \emph{potential functions} (also referred to as \textit{Lyapunov functions})~\citep{lyapunov1992general} for our analysis--see \S\ref{sec:dis} for details. Interestingly, the parameters of the algorithm determined--from first principles--by our potential function analysis \emph{exactly} satisfy the complicated recursive relations derived by Nesterov, thereby providing a simple alternative to his famous estimate sequence technique as a byproduct (\S\ref{sec:ensure}). 
Further, we also develop a fixed-point based analysis of how  accelerated convergence rates are derived from such complicated relations (\S\ref{sec:conv}), again providing a simple alternate to Nesterov's original analysis based on clever algebras.

Our new analysis is then applied to the Riemannian case (\S\ref{sec:riem} and \S\ref{sec:riemaccel}).
By carefully choosing potential function with the notion of projected distances, we nail down the main difficulty in the non-Euclidean case through the distortion in such distances (\S\ref{sec:riempoten}).
We then demonstrate that our analysis for a simplified setting already recovers the local acceleration results by \citep{zhang2018estimate}.
To tackle global acceleration, we establish a novel metric distortion inequality based on comparison theorems in Riemannian geometry (\S\ref{sec:valid}).
Our metric distortion result engenders a distortion rate that is accessible to algorithm at each iteration.
Finally, we show that such chosen tractable distortion rate decreases over iterations (\S\ref{sec:achieve}), thereby obtaining Theorem~\ref{thm:main:informal}.
\vspace*{-6pt}   
\subsection{Related work} \label{related}
Recently, a few efforts have been made to answer the main question of this paper. The first attempt is in~\citep{liu2017accelerated}, where the authors consider some nonlinear equations ((4) and (5) therein) and claim that the solutions to those equations yield accelerated schemes. However, it is \emph{a priori} unclear whether such equations are even feasible and whether solving them is tractable.

As for different approach, \cite{alimisis2019continuous} establish a Riemannian analogue of the  differential-equation approach to accelerated methods~\citep{su2014differential}. In particular, they propose and analyze several second-order ODEs on Riemannian manifolds that achieve accelerated convergence rates. Employing discretization results from the Euclidean case~\citep{betancourt2018symplectic,shi2019acceleration}, they also derive some first-order methods from the ODEs.  However, it is not clear whether the resulting methods achieve acceleration because such discretization techniques do \textit{not} directly yield Nesterov's accelerated method even in the Euclidean case.
   
The most concrete progress is the work~\citep{zhang2018estimate} that proves accelerated convergence of their algorithm, albeit only \emph{locally}: acceleration is shown to hold within a neighborhood whose radius vanishes as the condition number $L/\mu$ and the curvature bound $\kappa$ grow. Indeed, their work does not characterize how the algorithm behaves outside such a local neighborhood. This is in stark contrast with our \emph{global} acceleration result.  See \S\ref{sec:riempoten} for a detailed comparison.
  
On a less related note, \citet{siegel2019accelerated} develop accelerated gradient methods for functions on \emph{Stiefel manifolds}, by designing (i) a non-convex analogue of the adaptive restart methods due to~\citet{o2015adaptive}, and (ii) an analogue of momentum steps specifically for Stiefel manifolds. 
However, the accelerated performance is only verified \emph{numerically}. 

  
\vspace*{-6pt}   
\section{Warm up in the Euclidean case: alternative analysis of Nesterov's optimal method}
\label{sec:euclid}
Before we consider the Riemannian setting, we first illustrate our analysis in the Euclidean case.
In particular, we consider Nesterov's optimal method which is derived  based on an ingenious construction called an \emph{estimate sequence}~\citeyearpar[Ch. 2.2.1]{nesterov2018lectures}: For $t\geq 0$, the iterates are updated as\\[-10pt]
\begin{subequations} \label{nesterov}
	\begin{align}
	    x_{t+1} &\leftarrow y_t + \wg_{t+1} (z_t-y_t)\label{nest:0}\\
		y_{t+1} &\leftarrow  x_{t+1} -\sg_{t+1} \nabla f(x_{t+1})  \label{nest:1} \\
		z_{t+1} &\leftarrow  x_{t+1} + \wm_{t+1} (z_t-x_{t+1})  - \sm_{t+1} \nabla f(x_{t+1})\,.\label{nest:2}
	\end{align}
\end{subequations}
for given initial iterates $y_0=z_0\in \re^n$.  
This construction yields optimal first-order methods that achieve the lower bounds under the black-box complexity model~\citep{nemirovsky1983problem}. Note that  
without resorting to estimate sequences, the updates~\eqref{nesterov} can be also derived via the
\emph{linear coupling} framework of \citet{allen2014linear}. See Appendix~\ref{sec:linear} for details.

Despite its fundamental nature, there is a well-known puzzling aspect of Nesterov's construction. In order to guarantee the standing assumption of the estimate sequence technique~\citeyearpar[(2.2.3)]{nesterov2018lectures}, Nesterov's original analysis~\citeyearpar[page 87]{nesterov2018lectures} finds complicated recursive relations between parameters $\wg,\wm,\sg,\sm$ via some non-trivial algebraic ``tricks.'' These tricks are carried out in a fortuitous manner, obscuring the driving principle and the scope of the underlying technique. Notably, the work~\citep{zhang2018estimate} favors estimate sequences over other approaches, but still achieves only local acceleration. Therefore, toward our goal of obtaining global acceleration, we revisit Euclidean analysis of acceleration from first-principles.

In particular, we provide an alternative analysis of iteration~\eqref{nesterov} that sheds new light on understanding the scope of  Nesterov's original analysis. Our analysis employs a \textit{potential function},\footnote{Also known as \textit{Lyapunov function} in control theory or \textit{invariant} in theoretical computer science and mathematics}, a classical tool for studying the stability of dynamical systems~\citep{lyapunov1992general}, which has received a resurgence of interest due to its success in analyzing first-order optimization methods (see \S\ref{sec:dis} for details). Roughly, the convergence analysis based on potential-functions proceeds as follows:
\begin{enumerate}
  \setlength{\itemsep}{0pt}
\item  
  \textbf{Choose a potential function:} 
  First, choose a performance measure $\err_t$ defined for each step  that ``measures'' how close the iterates at step $t$ are to the optimal solution.  Having chosen $\err_t$, define the potential function as $\Phi_t:=A_t \err_t$ for some quantity $A_t$ to be determined.
\item {\bf Ensure potential decrease:} Choose an increasing sequence $A_t$ for which $\Phi_t$ is decreasing.  
\end{enumerate}
Once $\Phi_t$ is chosen as above, it immediately implies that $\err_t \leq \err_0/A_t$, yielding a convergence rate of $\OO{1/A_t}$ under the chosen performance measure. 
See \S\ref{sec:dis} for a discussion positioning our proposed potential function to related as well as different techniques for analyzing acceleration. 

\vspace*{-3pt}
\subsection{Choosing the potential function}
The key to potential function based analysis is to choose the ``correct'' performance measure.  
For an iterate $u_t$ at step $t$, two prototypical choices might be (i) the suboptimality  $\err_t=f(u_t)-f(x_*)$; and (ii) the distance to an optimal point $\norm{u_t- x_*}$. 
Indeed, many existing analyses correspond to choosing either one for performance measure, as explicitly observed in~\citep{bansal2019potential}.

For iteration~\eqref{nesterov}, it turns out that a weighted sum of the suboptimality $f(y_t)-f(x_*)$ and the distance $\norm{z_t-x_*}^2$ is the ``correct'' performance measure, i.e., we choose the potential  function as
\begin{align}
   \Phi_t := A_t\cdot \left( f(y_t) -f(x_*)\right) + B_t \cdot \norm{z_t-x_*}^2\label{def:poten}
\end{align}
for some $A_t>0$ and $B_t\geq 0$. By taking a weighted sum of the two measures, this performance measure does not require either  one to be monotonically decreasing over iterations.
This property, also known as \textit{non-relaxational property}, was indeed the main innovation  in Nesterov's illustrious paper~\citeyearpar{nesterov1983method}. The reason why we choose $y_t$ for the cost and $z_t$ for the distance will be clearer when we carry out the analysis. See Remark~\ref{rmk:why}.

The current form of the potential function~\eqref{def:poten} is not new; indeed it also appears in prior works~\citep{wilson2016lyapunov,diakonikolas2019approximate,bansal2019potential}, although with different motivations; see \S\ref{sec:dis} for precise details. Moreover, this potential also has an interpretation under the linear coupling framework of~\citep{allen2014linear}; see Appendix~\ref{sec:linear}.

\vspace*{-3pt}
\subsection{Potential difference calculations}
\label{sec:poten}
Having chosen the potential function~\eqref{def:poten}, the main goal now is to choose the parameters  $A_{t+1}$, $B_{t+1}$, $\wg_{t+1}$, $\wm_{t+1}$, $\sg_{t+1}$, $\sm_{t+1}$ so that the potential decreases, i.e., $\Phi_{t+1}-\Phi_t\leq 0$. To that end, we first express the potential difference $\Phi_{t+1}-\Phi_t$ more simply and derive a manageable upper bound using \emph{first principles}. Using defintion~\eqref{def:poten}, the difference $\Phi_{t+1}-\Phi_t$ can be split into two parts: 
\begin{align}
      &A_{t+1}\cdot \left(f(y_{t+1})-f(x_*) \right)  - A_t\cdot (f(y_t)-f(x_*)) \label{eq:gd}\\
    +\ &\ B_{t+1}\cdot  \norm{z_{t+1}-x_*}^2-B_t\cdot \norm{z_t-x_*}^2 \,.
    \label{eq:md} 
\end{align}

Since $\wg,\wm,\sg,\sm$ will only appear with index $t+1$, we drop their subscripts for simplicity. We first relate the terms for step $t+1$ with those for step $t$.
To do that, we reinterpret \eqref{nesterov}. Using the notation $\grad_{s \cdot \nabla }(x):=x-s\cdot \nabla$, the updates~\eqref{nest:1} and \eqref{nest:2} can be rewritten as 
\begin{align}
    \tag{\ref{nest:1}$'$} y_{t+1}&=\grad_{\sg \cdot\nabla f(x_{t+1})}(x_{t+1}) \label{nest:1'}\\
    \tag{\ref{nest:2}$'$}z_{t+1}&= \grad_{\sm \cdot\nabla f(x_{t+1})}( x_{t+1}+\wm (z_t-x_{t+1}))  \label{nest:2'}\,,
\end{align}
respectively. Now the difference between \eqref{nest:1'} and 
\eqref{nest:2'} is clear:  the first one is  an \emph{exact} gradient step in the sense that $\nabla=\nabla f(x)$, while the second one is \emph{inexact}.
Hence, in relating the terms for step $t+1$ with those for step $t$, we need to invoke different analyses for two different gradient steps.

Let us now review folklore analyses for gradient steps; the proofs are provided in Appendix~\ref{pf:folk}.  
First, when the gradient step is exact, the following result is well-known for an $L$-smooth cost:
\begin{proposition} \label{folk:grad}
Assume $\nabla=\nabla f(x)$, and let  $y =\grad_{s \cdot \nabla}(x)$. If $f$ is $L$-smooth, then the gradient step decreases cost:
$f(y)-f(x) \leq -s\left( 1 - Ls/2 \right)  \norm{\nabla}^2$.
\end{proposition} 
When the gradient step is inexact, one can only guarantee the following weaker result:
\begin{proposition}\footnote{Actually, the fact that this type of analyses (perhaps, better known as mirror descent analyses) can be applied to inexact gradient steps has brought about many successful applications in \textit{online optimization} (see e.g.~\citep{bubeck2011introduction}).} \label{folk :mirror}
  Let $z=\grad_{s\cdot \nabla }(x)$. Then, for any $x_*$, $\norm{z-x_*}^2- \norm{x-x_*}^2 = s^2 \norm{\nabla}^2 + 2s \inp{\nabla}{x_*-x}$, i.e.,   (inexact) gradient step decreases the distance to $x_*$ as long as direction $-\nabla$ is well aligned with the vector $x_*-x$ and has sufficiently small norm. 
 \end{proposition}
\begin{remark} \label{rmk:why}
These observations reveal why we use $y_t$ for the cost term and $z_t$ for the distance term in~\eqref{def:poten}: Proposition~\ref{folk:grad} deals with the cost, while Proposition~\ref{folk :mirror} deals with the distance.
\end{remark}

\noindent Now we apply Proposition~\ref{folk:grad} to \eqref{nest:1'} and Proposition~\ref{folk :mirror} to \eqref{nest:2'}. For clarity, we denote:
\begin{equation}
\label{eq:2}
\drop  := \sg  (1-L\sg  /2)\,,\quad\nabla := \nabla f(x_{t+1})\,,\quad X:=-(x_*-x_{t+1})\,,~\text{and}~ W:=z_t-x_{t+1}\,.
\end{equation}
With this notation, Propositions~\ref{folk:grad} and \ref{folk :mirror} imply:
$f(y_{t+1}) \leq f(x_{t+1})-\drop\norm{\nabla}^2$ and $\norm{z_{t+1}-x_*}^2 = \norm{X+\wm W}^2 + \sm^2 \norm{\nabla}^2 -2\sm \inp{\nabla}{X+\wm W}$. Plugging these two back into to \eqref{eq:gd} and \eqref{eq:md}, one can derive the following upper bound on $\Phi_{t+1}-\Phi_t$ in terms of the vectors $\nabla, X, W$ from \emph{first principles} (i.e., using only smoothness and (strong) convexity; see Appendix~\ref{app:derive}):
\begin{mdframed}
$\Phi_{t+1}-\Phi_t \le C_1\cdot \norm{W}^2 + C_2\cdot \norm{X}^2 +C_3\norm{\nabla}^2 +C_4 \cdot \inp{W}{X} +C_5\cdot \inp{W}{\nabla}+C_6\cdot \inp{X}{\nabla}, \eqnum \label{upper:1}$\\[6pt]
where $\begin{cases}
  C_1:=\wm^2 B_{t+1}-B_t  -  \frac{\mu}{2}\frac{\wg^2}{(1-\wg)^2}  A_t\,, &    C_2:=B_{t+1}-B_t - \frac{\mu}{2}(A_{t+1}-A_t)\,,\\
    C_3:= \sm^2 B_{t+1} - \drop \cdot A_{t+1}\,, &
    C_4:=2 \cdot \left(\wm B_{t+1} -B_t  \right)\,,\\
    C_5:= \frac{\wg}{1-\wg} A_t-2\wm\sm B_{t+1}\,,\quad\text{and} &
    C_6:=(A_{t+1}-A_t)-2\sm B_{t+1}\,.
    \end{cases}$
\end{mdframed}
The reader may have noticed that in~\eqref{eq:2} the vectors $\nabla$,$X$,$W$ are rooted at the same point $x_{t+1}$. This choice is deliberate, and will be crucial in the Riemannian setting, because there the vectors have to lie in the same \emph{tangent space} for their compatibility. See Appendix~\ref{pf:thm2}.

\vspace*{-3pt}
\subsection{Ensuring potential decrease} \label{sec:ensure}
Having established the bound~\eqref{upper:1}, our goal is to now choose $A_{t+1},B_{t+1},\wg,\wm,\sg,\sm$ given $A_t,B_t$ so that \eqref{upper:1} is non-positive (recall that we have dropped indices of $\wg_{t+1},\wm_{t+1},\sg_{t+1},\sm_{t+1}$). In general, it is difficult to ensure non-positivity of a symbolic expression; but since \eqref{upper:1} is a quadratic form, one avenue might be to turn it into a \emph{negative sum of squares} (``$-$SoS''). The simplest strategy to make it ``$-$SoS'' would be to choose parameters so that the coefficients $C_4$, $C_5$, $C_6$ of the cross terms are set to $0$, while $C_1,C_2,C_3\leq 0$. It turns out this strategy \emph{fully} determines the parameters.
\begin{list}{{\tiny $\blacksquare$}}{\leftmargin=1.5em}
  \vspace*{-4pt}
  \setlength{\itemsep}{0pt}
    \item \emph{Coefficients of cross terms characterize $\wg,\wm,\sm$ in terms of $A_{t+1},B_{t+1}$:} 
    From $C_6=0$, we get $\sm  = (A_{t+1}-A_t)/(2B_{t+1})$, and from $C_4=0$, we get $\wm = B_t/B_{t+1}$. Plugging these into $C_5=0$, we obtain the equation $\frac{\wg}{1-\wg}A_{t} = (A_{t+1}-A_t)B_t/B_{t+1}$. To summarize:\\[-8pt]
    \begin{equation}
      \label{stepsizes}
      \sm  = \tfrac{A_{t+1}-A_t}{2B_{t+1}}\,,\quad \wm = \tfrac{B_t}{B_{t+1}}\quad~\text{and}\quad\tfrac{\wg}{1-\alpha} = \tfrac{(A_{t+1}-A_t)B_t}{A_t B_{t+1}}\,. 
      \vspace*{-4pt}
    \end{equation}
  \item \emph{For a fixed $\sg$, coefficients of squared terms determines $A_{t+1},B_{t+1}$ based on given $A_t,B_t$:}
    Beginning with $C_3\leq 0$, we replace $\sm$ with the one from \eqref{stepsizes} to obtain:\\[-8pt]
   \begin{equation}
     \label{b:from:a}
     (A_{t+1}-A_t)^2/(4\drop \cdot  A_{t+1}) \leq B_{t+1}\,.
     \vspace*{-4pt}
\end{equation}
Plugging \eqref{b:from:a} into $C_2\leq 0$, we get an inequality only in terms of $A_{t+1}$ (assuming $\sg$ is fixed):\\[-8pt]
 \begin{equation}
   \label{ineq:a}
   (A_{t+1}-A_t)^2/(4\drop\cdot  A_{t+1}) -  (A_{t+1}-A_t)\tfrac{\mu}{2} \leq B_t \,.  
   \vspace*{-3pt}
 \end{equation}
 Recall from the potential function analysis that we need to choose $A_{t+1}$ as large as possible. It turns out that the largest possible $A_{t+1}$ satisfies~\eqref{ineq:a} as an equality  (hence~\eqref{b:from:a} as well). To see this, let us follow Nesterov's notation and consider the \emph{suboptimality shrinking ratio} $1-\cc:=\nicefrac{A_{t}}{A_{t+1}}$.\footnote{More precisely, Nesterov's estimate sequence analysis finds $\alpha_i\in (0,1)$ such that $f(y_t)-f(x_*) \leq \prod_{i=1}^t(1-\alpha_i)\cdot \left[ f(x_0)-f(x_*) + C\norm{x_0-x_*}^2\right]$ for some constant $C>0$, where $x_0$ is an initial iterate (see \citeyearpar[Theorem 2.2.1]{nesterov2018lectures}). Note that these $\alpha_i$'s exactly correspond to our suboptimality shrinking ratios.}
 With this ratio, inequality~\eqref{ineq:a} can be rewritten as:\\[-8pt]
 \begin{equation}
   \label{ineq:a2}
   \cc(\cc-2\mu\drop) /(1-\cc) \leq 4\drop\cdot  B_t/A_t\,.
   \vspace*{-4pt}
 \end{equation}
 In~\eqref{ineq:a2} note that the RHS is a nonnegative constant (assuming $\drop >0$ is already chosen) and the LHS is an increasing function on $[2\mu\drop,1)$ whose value is $0$ at $2\mu\drop$ and approaches $+\infty$ as $\cc\to 1$. Hence, the largest $\cc$ (equivalently, the largest $A_{t+1}$) satisfies \eqref{ineq:a2} (or equivalently, \eqref{ineq:a})  with equality. Consequently, this choice of $\cc$ also satisfies \eqref{b:from:a} with equality. One can then verify that this choice satisfies $\wm^2 B_{t+1}\leq B_t$ and hence implies $C_1\leq0$ (see Appendix~\ref{pf:thm1}).

 \item \emph{Lastly, choose $\sg$ from $(0,2/L)$:} Now the last variable to determine is $\sg$. The above calculations are valid as long as $\drop>0$, so we can arbitrarily choose $\sg$ in $(0,2/L)$. Note that most accelerated methods in the literature choose $\sg=1/L$ since it is the maximizer of $\drop$.
\end{list}
Furthermore, note that one can express the other variables $B_{t+1},\wg,\wm,\sm$ via equation~\eqref{stepsizes} and the equality version of \eqref{b:from:a} (see Appendix~\ref{pf:thm1}). Therefore, our findings can be summarized in the main result of this section as follows (after recovering the indices of $\wg,\wm,\sg,\sm$):

\begin{thmbox}[Parameter choice for potential decrease] \label{thm:euclidean}
  For $\sg_{t+1}\in(0,2/L)$, $\drop:= \sg_{t+1}(1-L\sg_{t+1}/2)$, and given $y_t,z_t$ and $A_t,B_t>0$, 
  \begin{enumerate}
    \setlength{\itemsep}{1pt}
  \item Compute $\cc_{t+1} \in [2\mu\drop, 1)$ satisfying $\frac{\cc_{t+1}(\cc_{t+1}-2\mu\drop) }{1-\cc_{t+1}} = 4\drop\cdot \frac{B_t}{A_t}. \eqnum \label{eq:recur}$
  \item Determine parameters based on $\cc$: $A_{t+1} =\frac{A_t}{1-\cc_{t+1}}$, $B_{t+1} = \frac{\cc_{t+1}^2}{1-\cc_{t+1}}\cdot \frac{A_t}{4\drop}$, $\wg_{t+1}=\frac{\cc_{t+1}-2\mu\drop}{1-2\mu\drop}$, $\wm_{t+1}=1- 2\mu\drop\cc_{t+1}^{-1}$, and $\sm_{t+1}=2\drop\cc_{t+1}^{-1}$.
  \end{enumerate}
  Then, $y_{t+1},z_{t+1}$ defined as per \eqref{nesterov} satisfy $\Phi_{t+1}\leq \Phi_t$ (see \eqref{def:poten}). 
  In other words, $$f(y_{t+1})-f(x_{*}) +\tfrac{\cc_{t+1}^2}{4\drop}\cdot\norm{z_{t+1}-x_*}^2 \leq (1-\cc_{t+1})\cdot \bigl[ f(y_t)-f(x_*)+\tfrac{B_t}{A_t}\cdot \norm{z_t-x_*}^2\bigr]\,.$$
\end{thmbox}
\begin{proof}
  An immediate consequence of the above calculations; See Appendix~\ref{pf:thm1} for details.
\end{proof}

A noteworthy outcome of the above analysis is that the parameter choices in Theorem~\ref{thm:euclidean} \emph{exactly match} those of Nesterov's ``General Scheme for Optimal Method''~\citeyearpar[(2.2.1)]{nesterov2018lectures}. Hence, our approach actually recovers Nesterov's optimal method that encompasses both strongly and non-strongly convex costs, without resorting to the estimate sequence technique. 
Another important byproduct of our analysis is the convergence of $z_t$ to $x_*$ for $\mu>0$ (in which case, $\cc >0)$, a property proved via additional analysis in the literature (see e.g.,~\citep[Corollary 1]{gasnikov2018universal}). More importantly, this convergence plays a crucial role in the Riemannian setting. See \S\ref{sec:achieve}.

Observe that upon applying Theorem~\ref{thm:euclidean} recursively, we can deduce that 
\begin{align} \label{rate:conv}
  f(y_t)-f(x_*)  = \OO{(1-\cc_1)(1-\cc_{2})\cdots (1-\cc_{t})}.
\end{align}
Thus, to identify the convergence rate of iteration~\eqref{nesterov} with parameters chosen via Theorem~\ref{thm:euclidean}, we only need to study how the sequence $\{\cc_t\}$ evolves. This evolution is the focus of the next subsection.

\subsection{Identifying the convergence rate of \eqref{nesterov}: a simple analysis based on fixed-point iteration}
\label{sec:conv}
For studying the evolution of $\cc_t$, we consider the strongly convex case ($\mu>0$) and assume that $\sg_t$ is fixed to be a constant $\sg\in(0,2/L)$. This assumption is not stringent as most accelerated methods in the literature choose $\sg_t\equiv 1/L$.

Our approach below offers an alternative to its counterpart in Nesterov's book~\citeyearpar[Lemma 2.2.4]{nesterov2018lectures}: In contrast to Nesterov's analysis based on clever algebraic manipulations, our approach \emph{directly} analyzes the evolution of the sequence by drawing a connection to \emph{fixed point iterations}.
\begin{wrapfigure}{r}{0.3\textwidth}  
\centering
    \begin{tikzpicture}[scale=0.6]
\begin{axis}[
    xmin=0, xmax=1,
    ymin=0, ymax=1,
    ]
\addplot [
    domain=0:1, 
    samples=100, 
    color=red, 
    ultra thick
]
{x^2};
\addplot [
    domain=0:0.9, 
    samples=100, 
    color=blue, 
    ultra thick
]
{x*(x-0.25)/(1-x)};
\addplot[
    color=black 
    ]
    coordinates {
    (0.9,0.81)(0.6625,0.81)(0.6625,0.4389) (0.5748,0.4389) (0.5748, 0.3304)(0.5360,0.3304)(0.5360,0.2873)
    }; 
\end{axis}
\end{tikzpicture}
    \caption{\footnotesize{An illustration of the evolution of \eqref{recur:xi} for $2\mu\drop =0.25$.
    We plot $\cho=\frac{v(v-2\mu\drop )}{(1-v)}$  in \textcolor{blue}{blue} and  $\cht(v)=v^2$  in \textcolor{red}{red}.}}
    \label{fig:1}
\end{wrapfigure}
Moreover, this approach will generalize better to the more complicated Riemannian setting.
As a byproduct of our simple approach, we were able to remove the technical condition on $\cc_0$ required by Nesterov's analysis. See Remark~\ref{rmk:improve:nesterov}.

First, let us pinpoint the recursive relation satisfied by $\cc_t$.
Applying Theorem~\ref{thm:euclidean} for $t\leftarrow t-1$, we have $\frac{B_t}{A_t}=\frac{\cc_t^2}{4\drop}$. Hence, \eqref{eq:recur} implies the following \emph{nonlinear} recursive relation on $\cc_t$'s:
\begin{align}  \cc_{t+1}(\cc_{t+1}-2\mu\drop) /(1-\cc_{t+1}) = \cc_{t}^2\,. \label{recur:xi}
\end{align}
Now, our objective is to characterize the evolution of $\cc_t$ from \eqref{recur:xi}.
Intuitively, \eqref{recur:xi} can be construed as the recursive relation for computing the root of $\cho(v)=\cht(v)$, where $\cho(v):= \frac{v(v-2\mu\drop )}{(1-v)}$ and $\cht(v):=v^2$.
Since the root is equal to $v=\sqrt{2\mu\drop}$, one can guess that $\cc_t\to \sqrt{2\mu\drop}$. 
See Figure~\ref{fig:1} for illustration. 
The following lemma  confirms this guess.

\begin{lembox}[Evolution of \eqref{recur:xi}] \label{lem:conv}
  For an arbitrary initial value $\cc_0\geq 0$, let $\cc_t$ ($t\geq 1$) be the sequence of numbers defined as per \eqref{recur:xi}.
  Then, $\cc_t\in[2\mu\drop,1)$ for all $t\geq 1$. 
  Furthermore, if \\
  {\footnotesize $\begin{cases} \cc_0>\sqrt{2\mu\drop}\,,\\ \cc_0 = \sqrt{2\mu\drop}\,,\\ \cc_0<\sqrt{2\mu\drop}\,, \end{cases}$} then {\footnotesize $\begin{cases} \cc_t \searrow \sqrt{2\mu\drop} \text{ as $t\to \infty$}\,.\\ \cc_t \equiv~ \sqrt{2\mu\drop}\,. \\ \cc_t \nearrow \sqrt{2\mu\drop}\text{ as $t\to \infty$}\,. \end{cases}$} In particular, the convergences are \emph{geometric}.
\end{lembox}
\begin{proof}
  The proof and the formal statement (Lemma~\ref{lem:appen}) are provided in Appendix~\ref{pf:conv}.
\end{proof}

Lemma~\ref{lem:conv} delivers the desired accelerated convergence rate:
\begin{corobox} \label{coro:accel}
If  $\cc_0\geq  \sqrt{2\mu\drop}$, then 
$f(y_t)-f(x_*) =O(\prod_{i=1}^t (1-\sqrt{2\mu\drop})) =O(\exp ( - t\sqrt{2\mu\drop}))$.
In particular,  setting $\sg =1/L$,  $f(y_t)-f(x_*) =O(\exp ( - t\sqrt{\nicefrac{\mu}{L}}))$.
\end{corobox}
\begin{remark}[Removing technical conditions in Nesterov's analysis] \label{rmk:improve:nesterov}
Nesterov's original analysis requires a technical condition on the initial value $\cc_0$: $\sqrt{\nicefrac{\mu}{L}}\leq \cc_0 \leq \nicefrac{(2(3+\mu/L))}{(3+\sqrt{21+4\mu/L})}$ \citep[(2.2.21)]{nesterov2018lectures}.
In contrast, our analysis reveals that the upper bound on $\cc_0$ is not needed; the lower bound is also not needed in the sense that $\cc_t$ converges to $\sqrt{\mu/L}$, the accelerated rate.
\end{remark}

\vspace*{-12pt}
\section{Generalization to the non-Euclidean case: Riemannian potential function analysis}
\label{sec:riem}
\vspace*{-4pt}
This section develops the first key ingredient towards obtaining our main theorem (Theorem~\ref{thm:main}), namely, Theorem~\ref{thm:riem} that is a Riemannian analogue of Theorem~\ref{thm:euclidean}. We begin by introducing some preliminaries and stating the Riemannian analogue of iteration~\eqref{nesterov}.

\vspace*{-4pt}
\subsection{Riemannian geometry and Riemannian analogue of Nesterov method}
\label{sec:riemnest}
\vspace*{-4pt}
We recall below some basic concepts from Riemannian geometry, and defer to textbooks (e.g., \citep{jost2008riemannian,burago2001course}) for an in depth introduction. A \emph{Riemannian manifold} is a smooth manifold $M$ equipped with a smoothly varying inner product $\inp{\cdot}{\cdot}_x$ (called the \emph{Riemannian metric}) defined for each $x \in M$ on the tangent space $T_xM$.
One can define geometric concepts such as angles, length of curves, and surface areas on a Riemannian manifold $M$. With the concept of length of curves, one can introduce a distances function $d$ on $M$ and consequently view $(M,d)$ as a metric space. Length also allows us to define analogues of straight lines, namely \emph{geodesics}: A curve is a \emph{geodesic} if it is locally distance minimizing. The notions of curvature can be also defined in $M$. We focus on \emph{sectional curvature}, which characterizes the notion of curvature by measuring Gaussian curvatures of $2$-dimensional submanifolds of $M$. We make the following key assumption on curvature:
\begin{assump} We assume that the sectional curvature is lower bounded by $-\kappa$ for some nonnegative constant $\kappa$. This is a widely used standard assumption in Riemannian geometry; for a textbook treatment dedicated to this assumption see e.g.,~\citep[Chapter 10]{burago2001course}.
\end{assump}

\noindent\emph{Operations on Riemannian manifolds.} One can define analogues of vector addition and subtraction in Riemannian manifolds via the concept of \emph{exponential map}. 
 An exponential map $\expm_x: T_x M \to M$ maps $v\in T_x M$ to $g(1)\in M$ for a geodesic $g$ with $g(0)=x$ and $g'(0)=v$. We remark that it is always well-defined locally.
Notice that $\exm{x}{v} \in M$ is an analogue of vector addition ``$x+v$.'' Similarly, the inverse map $\lm{x}{y} \in T_x M$ is a non-Euclidean analogue of vector subtraction ``$y-x$.''  For $\expm_x^{-1}$ to be well-defined for each $x$, we assume that any two points on $M$ are connected by a unique geodesic. 
This property is called \emph{uniquely geodesic}, and it is valid  locally for general Riemannian manifolds and globally for Hadamard manifolds (Riemannian manifolds with global non-positive sectional curvatures).  
We assume further that $\expm_x,\expm_x^{-1}$ can be comptuted at each $x$, which is in fact the case for many widely used matrix manifolds~\citep{absil2009optimization}.   
 
\noindent\emph{Convexity on Riemannian manifolds.} The notion of convexity can be extended to Riemannian manifolds using geodesics where a convex combination of two points is defined along the geodesic connecting them. This generalized notion of convexity is called \emph{geodesic convexity} (\emph{g-convexity} for short)~\citep{gromov1978manifolds}.  In particular, one can define geodesic-smoothness and (strong) g-convexity for a function $f:M\to \re$ akin to their Euclidean counterparts.\\[-14pt]
\begin{assump}
We assume that the cost function $f$ is geodesically $L$-smooth and $\mu$-strongly convex (formal definitions in Appendix~\ref{pf:thm2}; see also \citep[Section 2.3]{zhang2016first}).\vspace*{-4pt}
\end{assump}
Now we are ready to generalize Nesterov's method~\eqref{nesterov} to Riemannian manifolds. 
Using the above noted Riemannian analogues of vector addition and subtractions, \eqref{nesterov} becomes
\begin{subequations} \label{r:nesterov}
	\begin{align}
	    x_{t+1} &\leftarrow \exm{y_t}{\wg_{t+1}\lm{y_t}{z_t}}\label{r:nest:0}\\
		y_{t+1} &\leftarrow  \exm{x_{t+1}}{-\sg_{t+1} \nabla f(x_{t+1})}  \label{r:nest:1} \\
		z_{t+1} &\leftarrow \mathrm{Exp}_{x_{t+1}}\bigl(\wm_{t+1} \lm{x_{t+1}}{z_t}  - \sm_{t+1} \nabla f(x_{t+1})\bigr)\,.\label{r:nest:2}
	\end{align}
\end{subequations}
Note that updates~\eqref{r:nest:1} and \eqref{r:nest:2} are well-defined since $\nabla f(x_{t+1})$ lies in the tangent space $T_x M$. Having formalized the Riemannian version of~\eqref{nesterov}, we study its potential function analysis.

\vspace*{-4pt}
\subsection{Riemannian potential function analysis and metric distortion} 
\label{sec:riempoten}
\vspace*{-4pt}
Since iteration~\eqref{r:nesterov} is a direct analog of its Euclidean counterpart~\eqref{nesterov}, one may be tempted to study the potential function
$\Psi_t:= A_t\cdot (f(y)-f(x_*))+B_t\cdot d(z_t,x_*)^2$
that is a direct analog of~\eqref{def:poten}. 
However, it turns out that the following less direct potential function is much more advantageous:\\[-14pt]
\begin{align}
  \Psi_t := A_t\cdot \left( f(y_t) -f(x_*)\right) + B_t \cdot  \norm{\lm{x_t}{z_t}-\lm{x_t}{x_*}}^2\,.\label{r:def:poten}
  \vspace*{-4pt}
\end{align}%
The distance term in~\eqref{r:def:poten} is preferable to $d(z_t,x_*)^2$ because it lets us use Euclidean geometry (since it is defined on the tangent space $\T_{x_t} M \cong \re^n$) to control it. To simplify notation, we define: 
\begin{definition}[Projected distance] \label{def:pd}
  For any three points $u,v,w \in M$, the projected distance between $v$ and $w$ relative to $u$ is defined as $\DD{u}{v}{w}:= \norm{\lm{u}{v}-\lm{u}{w}}$.
\end{definition}

However, there is one fundamental hurdle \emph{inherent} to comparing distances in the Riemannian setting: we need to handle the \emph{incompatibility} of metrics between two different points.
A key advantage of the potential function analysis is that one only needs to focus on comparing the distances appearing in \emph{adjacent} terms, namely $\Psi_t$ and $\Psi_{t+1}$, which simplifies the argument considerably. Motivated by~\eqref{r:def:poten}, we define the following quantity for comparing distances:
\begin{definition}[Valid distortion rate] 
\label{def:val}
We say $\ddr{t}$ is a \textit{valid distortion rate} at  iteration $t\geq 1$ if the following inequality holds:  $\DD{x_{t}}{z_{t-1}}{x_*}^2 \leq \ddr{t}\cdot \DD{x_{t-1}}{z_{t-1}}{x_*}^2$.
\end{definition}
Assuming the existence of valid distortion rates at \emph{each} iteration, we can analyze iteration~\eqref{r:nesterov} analogously to the analysis in~\S\ref{sec:poten} and \S\ref{sec:ensure}, to obtain the main result of this section:

\begin{thmbox}[Riemannian analogue of Theorem~\ref{thm:euclidean}] \label{thm:riem}
  Assume that $\ddr{t}$ is a valid distortion rate at iteration $t$.
  For $\sg_{t+1}\in(0,2/L)$, $\drop:= \sg_{t+1}(1-L\sg_{t+1}/2)$, and given $y_t,z_t$ and $A_t,B_t>0$,
  \begin{enumerate}
  \item Compute $\cc_{t+1} \in [2\mu \drop, 1)$ satisfying
    $ \frac{\cc_{t+1}(\cc_{t+1}-2\mu\drop) }{1-\cc_{t+1}} =\frac{ 4\drop}{\pmb{\ddr{t+1}}}\cdot \frac{B_t}{A_t}\,. \eqnum\label{r:eq:recur}$
  \item Compute $A_{t+1},B_{t+1},\wg_{t+1},\wm_{t+1},\sm_{t+1}$  as in Theorem~\ref{thm:euclidean}.
  \end{enumerate} 
  Then, $y_{t+1},z_{t+1}$ generated via iteration~\eqref{r:nesterov} satisfy $\Psi_{t+1}\leq \Psi_t$ (see \eqref{r:def:poten}). In other words,  
  $$f(y_{t+1})-f(x_{*}) +\tfrac{\cc_{t+1}^2}{4\drop}\cdot\DD{x_{t+1}}{z_{t+1}}{x_*}^2 \leq (1-\cc_{t+1})\cdot \bigl[ f(y_t)-f(x_*)+\tfrac{B_t}{A_t}\cdot \DD{x_{t}}{z_{t}}{x_*}^2\bigr]\,.$$
\end{thmbox}
\begin{proof}
The proof resembles the arguments in \S\ref{sec:poten} and \S\ref{sec:ensure}, modulo the appearance of valid distortion rates in \eqref{r:eq:recur}. See Appendix~\ref{pf:thm2} for details.
\end{proof}
As before, we can deduce from Theorem~\ref{thm:riem} the suboptimality gap bound~\eqref{rate:conv}.
Hence, to identify the convergence rate we only need to determine the evolution of $\{\cc_t\}$. We provide an illustrative example below, before moving onto the full accelerated algorithm in \S\ref{sec:riemaccel}. 

\paragraph{Illustrative example: constant distortion rate.} Consider the simplified case where $\ddr{t}\equiv \ddr{}\geq 1$ for all $t\geq 0$.
Under this constant distortion condition, similarly to recursion~\eqref{recur:xi}, we can obtain a recursive relation on $\{\cc_t\}$ by choosing $\sg_t\equiv\sg\in(0,2/L)$:\\[-10pt]
  \begin{equation}\label{r:recur:xi}
      \cc_{t+1}(\cc_{t+1}-2\mu\drop)/(1-\cc_{t+1}) = \cc_{t}^2/\pmb{\ddr{}} \,. 
   \vspace*{-2pt}
 \end{equation}
Here the only difference relative to \eqref{recur:xi} is now the RHS is divided by $\ddr{}$. 
Analogously to Lemma~\ref{lem:conv}, one can establish geometric convergence of $\cc_t$ to the fixed point $\cc(\ddr{})$ of equation~\eqref{r:recur:xi} (see  Lemma~\ref{lem:appen}).
Solving for $\cc(\ddr{})$ explicitly, we obtain the following analogue of Corollary~\ref{coro:accel}: 
\begin{corobox} \label{r:coro:accel}
If  $\cc_0\geq  \cc(\ddr{}):=\nicefrac{\sqrt{(\ddr{}-1)^2+8\ddr{}\mu\drop }}{2}-\nicefrac{(\ddr{}-1)}{2}$, then the following convergence rate holds: $f(y_t)-f(x_*) =O\big(\prod_{i=1}^t (1-\cc(\ddr{}))\big) =O\big(\exp (- t\cdot \cc(\ddr{}))\big)$.
In particular,  setting $\sg =1/L$,  $f(y_t)-f(x_*) =O\bigl((\exp \bigl(- \frac{t}{2}\{  \sqrt{(\ddr{}-1)^2+4\ddr{}\nicefrac{\mu}{L} } -\frac{t}{2}(\ddr{}-1)\}  \bigr)\bigr)$.
\end{corobox}
A notable aspect of Corollary~\ref{r:coro:accel} is that it characterizes a trade-off between the metric distortion and the convergence rate of the resulting algorithm. This point is elaborated by the following remark:
\begin{remark}[Properties of $\cc(\ddr{})$] \label{rmk:constant}
When there is no distortion, i.e., $\ddr{}=1$., then $\cc(1)=\sqrt{2\mu\drop}$ since \eqref{r:recur:xi} becomes \eqref{recur:xi}. Moreover, one can verify that $\cc(\ddr{})$ is (strictly) decreasing in $\ddr{}$, implying that the algorithm's performance gets worse as the distortion gets severer (see Appendix~\ref{appen:just} for verification). Hence, $\cc(\ddr{})>\lim_{\ddr{}\to \infty}\cc(\ddr{}) =2\mu\drop$ for all $\ddr{}>1$, implying that the convergence rate  is always \textbf{\emph{strictly}} better than gradient descent no matter how severe the distortion is. 
\end{remark}

Interestingly, our  illustrative analysis  already recovers the local acceleration result of \citet{zhang2018estimate}. More specifically, they showed that if the distance between initial iterates and the global optimum is bounded by $\nicefrac{1}{20}\cdot\kappa^{1/2}(\nicefrac{L}{\mu})^{-3/4}$ (see Appendix F therein), then the distortion is bounded by $\ddr{}=1+\nicefrac{1}{5}\cdot (\nicefrac{L}{\mu})^{-1/2}$. Simplifying the value $\cc(\ddr{})$ for this choice of $\ddr{}$, we obtain the following strengthening of their main result~\citep[Theorem 3]{zhang2018estimate}:
\begin{corobox}[Local acceleration]
  Let $\ddr{}=1+\frac{1}{5}\cdot (\nicefrac{\mu}{L})^{1/2}$,  $\sg =1/L$ and $\cc_0 \geq\cc(\ddr{})$.
  Then, assuming $\dd{x_0}{x_*}\leq \frac{1}{20}\cdot\kappa^{1/2}(\nicefrac{\mu}{L})^{3/4}$,
  we have $f(y_t)-f(x_*)  =O(\exp ( -  \frac{9}{10}t \sqrt{\nicefrac{\mu}{L}}))$.
  In particular,  $\cc_t= \cc(\ddr{})$ for all $t\geq0$, recovers \citep[Algorithm 2]{zhang2018estimate}.
\end{corobox}

\vspace*{-12pt}	
\section{Riemannian Accelerated Gradient Method}
\label{sec:riemaccel}
Thus far, the analysis assumed existence of valid distortion rates. Now the question is: \emph{are valid distortion rates available to the method?} We provide a positive answer below and therewith propose a new Riemannian accelerated gradient method. For clarity, we will focus on Hadamard manifolds;  the development for non-Hadamard manifolds is analogous and is provided in Appendix~\ref{sec:nonhada}.

\vspace*{-8pt}	
\subsection{Valid distortion rates and Riemannian accelerated gradient method} \label{sec:valid} 
\vspace*{-3pt}	
We estimate metric distortion by first invoking a classical comparison theorem of~\citet{rauch1951contribution}.
\begin{proposition}\label{prop:rauch}
  Let $x,y,z \in M$, a Riemannian manifold with curvature lower bounded by $-\kappa<0$. Let $S_\kappa(r):=\bigl(\frac{\sinh(\sqrt{\kappa}r)}{\sqrt{\kappa}r}\bigr)^2$; then, we have
  $\dd{y}{z}^2 \leq S_\kappa(\max\{\dd{x}{y},\dd{x}{z} \}) \cdot \DD{x}{y}{z}^2$.
\end{proposition}
\begin{proof}
   A direct consequence of the Rauch comparison theorem; see Appendix~\ref{appen:geo} for completeness. 
\end{proof}
Proposition~\ref{prop:rauch} applied to $x_t$,$z_t$,$x_*$ yields $\dd{z_t}{x_*}^2\leq S_\kappa(\max\{ \dd{x_t}{z_t},~\dd{x_t}{x_*} \} )\cdot \DD{x_t}{z_t}{z_*}^2$. Next, from Topogonov's comparison theorem (see e.g.~\citep[Section 6.5]{burago2001course}) we know that  $\DD{x_{t+1}}{z_t}{x_*}\leq \dd{z_t}{x_*}$ since $M$ is a Hadamard manifold.
Combining these two inequalities, we see that  $\delta_{t}=S_\kappa(\max\{ \dd{x_t}{z_t},~\dd{x_t}{x_*} \} )$ is a valid distortion rate. However, the issue with this argument is that the distortion rate depends on $d(x_t,x_*)$, which is in general unavailable to the algorithm. We overcome this fundamental difficulty by developing a new distortion inequality.
\begin{lembox}[Improved metric distortion inequality] \label{lem:betterdist}
  Let $x,y,z$ be points on Riemannian manifold $M$ with sectional curvatures lower bounded by $-\kappa<0$.
  Then for $\vd:\re_{\geq 0} \to \re_{\geq 1}$ defined as\\[3pt]
  $\vd(r) :=  
  \begin{cases}
    \max\Bigl\{ 1+4\bigl(\frac{\sqrt{\kappa}r}{\tanh(\sqrt{\kappa}r)} -1\bigr),~\bigl(\frac{\sinh(2\sqrt{\kappa}\cdot r)}{2\sqrt{\kappa}\cdot r}\bigr)^2 \Bigr\}, & \text{if $r>0$,}\\
    1, & \text{if $r=0$,} 
  \end{cases} \eqnum \label{def:vd}$\\
  the following inequality holds: $\dd{y}{z}^2\leq \vd(\dd{x}{y})\cdot \DD{x}{y}{z}^2$.
\end{lembox} 
\begin{proof}
  The proof uses Proposition~\ref{prop:rauch} and a Riemannian trigonometric inequality due to \cite[Lemma 6]{zhang2016first}. See Appendix~\ref{appen:geo} for a formal statement and the proof.
\end{proof}%
%
Note that $\vd$  behaves similarly to $S_{\kappa}$. Most importantly, $\lim_{r\to 0+}\vd(r) =1$, implying that the effect of distortion diminishes as the distance decreases. Hence, one can essentially regard Lemma~\ref{lem:betterdist}  as a version of  Proposition~\ref{prop:rauch} in which the term $\max\{ \dd{x}{y}, \dd{x}{z} \}$ is replaced with $\dd{x}{y}$. 
Thanks to Lemma~\ref{lem:betterdist}, now we have $\vd(\dd{x_t}{z_t})$ as a valid distortion rate, which is \emph{accessible} to the algorithm at iteration $t$. Therefore, we propose the following algorithm:

\begin{alg}[Riemannian accelerated gradient method] \label{alg:1}

  {\bf Input:} $x_0=y_0=z_0\in M$; constant $\cc_0>0$; $\sg\in(0,2/L)$; $\drop:= \sg(1-L\sg/2)$; integer $T$.
  
{\bf for $t=0,1,2,\dots,T$:}

\quad Compute the distortion rate $\drr{t+1}:= \vd(\dd{x_t}{z_t})$ as per \eqref{def:vd}.

\quad Find $\cc_{t+1}\in[2\mu\drop,1)$ such that $\frac{\cc_{t+1}(\cc_{t+1}-2\mu\drop)}{1-\cc_{t+1}}= \frac{1}{\drr{t+1}}\cdot  \cc_t^2$.

\quad Compute $\wg_{t+1} :=  \frac{\cc_{t+1}-2\mu\drop}{1-2\mu\drop}$, $\wm_{t+1}:= 1-2\mu\drop \cc_{t+1}^{-1}$, and  $\sm_{t+1}:= 2\drop \cc_{t+1}^{-1}$.

\quad Update the next step iterates as per \eqref{r:nesterov} with $\sg_{t+1}:=\sg$.

{\bf end for}
\end{alg}
\subsection{Convergence rate analysis of the proposed method}\label{sec:achieve}
Having proposed the algorithm, our final task is to analyze its convergence rate. From Remark~\ref{rmk:constant}, we know the algorithm achieves a \emph{full} acceleration when  $\ddr{t}$ is close to $1$.
Due to the property $\lim_{r\to 0+}\vd (r) = 1$, one therefore needs to show that $\dd{x_t}{z_t}$ is close to $0$.
Although $\dd{x_0}{z_0}=0$, one can quickly notice that this is obviously not true for $t\geq 1$.

Now one natural follow-up question is whether $\dd{x_t}{z_t}$ shrinks over iterations.
As we have seen in \S\ref{sec:ensure},  the convergence of the iterates to the optimal point is a direct consequence of our potential function analysis.
Similarly, one can immediately see that $\DD{x_t}{z_t}{x_*}\to 0$.
It turns out that from this shrinking projected distance, one can also deduce $\dd{x_t}{z_t}\to 0$ under mild conditions:
\begin{lembox}[Shrinking $\dd{x_t}{x_t}$] \label{lem:shrink:d}
  Assume $0<\mu$; let $\init :=  f(x_0)-f(x_*) + \frac{1}{4\drop} \cc_{0}^2\cdot \dd{x_0}{x_*}^2$.
  If $1 < \sg L<2-\cc_{t}$ and $\cc_{t }>2\mu\drop$ hold at iteration $t\geq 1$, then Algorithm~\ref{alg:1} satisfies: $\dd{x_{t}}{z_{t}} \leq   \co \bigl[\init\prod_{j=1}^{t-1}(1- \cc_j)\bigr]^{1/2}$, where $ \co>0$ is a constant depending only  on $\mu,L,\sg$.
\end{lembox}
\begin{proof}
  The proof relies on elementary geometric inequalites (see Appendix~\ref{pf:shrink}).
\end{proof}
Note that the assumption $\sg L\in (1,  2-\cc_{t }]$ can be roughly read as ``$\sg L \in (1 ,  2-\sqrt{\mu/L}]$'' because Remark~\ref{rmk:constant} ensures that $\cc(\ddr{}) \leq \sqrt{2\mu\drop} <\sqrt{\mu/L}$ for all $\ddr{}\geq 1$. More precisely, since $\cc_t$ quickly converges to the fixed point, one can easily ensure  $\cc_t \leq \sqrt{\mu/L}$ after few iterations. Formalizing this argument, we finally obtain our main theorem (which formalizes Theorem~\ref{thm:main:informal}):
\begin{thmbox}[Eventual full acceleration of Algorithm~\ref{alg:1}] \label{thm:main}
  Assume that  $ 0<\mu<L$. 
  For any $\cc_0 >0$, $\sg L\in(1,2-\sqrt{\mu/L}]$, let
  $\drop:= \sg(1-L\sg/2)$ and 
  $\rr:=1-\frac{8\mu\drop}{5+\sqrt{5}} \in (0,1)$.
  Then, Algorithm~\ref{alg:1} satisfies the following accelerated convergence:\\[-18pt]   
  \begin{align}
    f(y_t)-f(x_*)=\OO{(1-\cc_1)(1-\cc_2)\cdots (1-\cc_t)}\,, \label{accel:rate}
    \vspace*{-3pt}
  \end{align}
  where $\{\cc_t\}$ is a sequence such that (i) $\cc_{t}> 2\mu\drop$ $\forall t\geq 0$ and (ii)  for all $\eps>0$, $|\cc_t-\sqrt{2\mu\drop}|\leq \eps$ whenever $t= \Omega\left(\frac{\log(1/\eps)}{\log(1/\rr)}\right)$, where the constant involved in $\Omega(\cdot)$ depends only on $\mu,L,\sg,\kappa$. 
\end{thmbox}
\begin{proof}
  \eqref{accel:rate} is immediate from Theorem~\ref{thm:riem}. For the convergence of $\{\cc_t\}$, see Appendix~\ref{pf:main}.
\end{proof}
Since $\drop \to 1/(2L)$ as $\sg \to 1/L$, one can achieve the convergence rate \emph{arbitrarily} close to the full acceleration rate by choosing $\sg$ bigger but sufficiently close to $1/L$.
This concludes our main results.

\section{Comparison with other potential function analyses} \label{sec:dis}
In this section, we compare existing potential function analyses with our approach. 
For a survey, see e.g.~\citep[Appendix B]{taylor2019stochastic}.

\noindent\textbf{Other approaches to \eqref{def:poten}.}
The potential function~\eqref{def:poten} has appeared in prior work on accelerated methods, corroborating its \emph{suitability}. Compared with our analysis, the main difference is that the existing analyses either work for (i) the case $\mu=0$, or (ii) just the fixed-step case $\cc_t=\cc(\ddr{})$.
We highlight that our analysis is the first to recover--from first principles--Nesterov's general scheme that smoothly interpolates the cases $\mu=0$ and $\mu>0$.
Moreover, our analysis allows $\cc_t$ to vary, which is crucial in the Riemannian case where the recursive relation changes over iterations.

Function~\eqref{def:poten} appears in~\citep[Proposition 4]{wilson2016lyapunov} within the context of a continuous dynamics approach to acceleration. That work studies methods for discretizing \emph{accelerated} ODEs derived in~\citep{su2014differential,wibisono2016variational} to transform the continuous dynamics into discrete methods. In that context, they show that \eqref{def:poten} is a discretization of a \emph{canonical} Lyapunov function. Another appearance is in \citep{diakonikolas2019approximate}, where they extend the continuous dynamics view via an \emph{approximate duality gap technique}. Roughly, to analyze a first-order method, they consider an upper bound $U_t$ and a lower bound $L_t$ on the optimal value $f(x_*)$. Their analysis then proceeds by showing the gap $G_t:=U_t-L_t$ diminishes with the rate $\alpha_t$, i.e., $\alpha_t G_t$ is decreasing, which corresponds to  showing $A_t\err_t$ is decreasing in our language (\S\ref{sec:euclid}).
Although motivated mostly for continuous dynamics, their techniques cover discrete methods with some modifications. In particular, their choice of $G_t$ for accelerated method corresponds  to~\eqref{def:poten} (see \S4.2 therein).

Yet another appearance of \eqref{def:poten} is~\citep[(5.50)]{bansal2019potential}, wherein the motivation was to modify the potential function analyses for gradient descent to cover accelerated methods. They propose the idea of running \emph{two different} gradient steps and linearly combining them to achieve desired accelerated convergence. Following their argument, it turns out \eqref{def:poten} is the right choice. Indeed, their approach bears resemblance to the \emph{linear coupling} framework \citep{allen2014linear}, in which \eqref{def:poten} has even more \emph{canonical} interpretations; see Appendix \ref{sec:linear}.

\noindent\textbf{SDP-based approaches.} 
The primary distinction between our approach and most SDP-based approaches~\citep{drori2014performance,lessard2016analysis,taylor2018lyapunov,taylor2019stochastic} is that our analysis is \emph{analytical}, whereas the analyses therein are \emph{numerical}. More specifically, the existing works require \emph{numeric values} of parameters (e.g., $\wg,\wm,L,\mu$) because they find  suitable potential functions via solving SDPs. Note that one \emph{cannot} solve SDPs unless the numeric coefficients are given. Abstractly, our choice of parameters in Theorem~\ref{thm:euclidean} can be interpreted as an \emph{analytical} solution to the \emph{symbolic} versions of SDPs formulated in the prior works.

Notable exceptions are \citep{kim2016optimized,hu2017dissipativity,safavi2018explicit,cyrus2018robust,aybat2018robust}, in which small SDPs are solved \emph{analytically}. Specifically, some optimized step sizes for Nesterov's method are derived via solving small SDPs explicitly in~\citep{kim2016optimized,safavi2018explicit}; robust versions of gradient methods are derived analytically via classical control-theoretic arguments 
in \citep{cyrus2018robust,aybat2018robust}, and Nesterov's method is reinterpreted using \emph{dissipativity theory} in \citep{hu2017dissipativity}. Indeed borrowing the dissipativity interpretation from \citep{hu2017dissipativity}, one can interpret our calculations in \S\ref{sec:ensure} as finding a \emph{analytic} solution to the dissipation inequality (Theorem 2 therein).

\section{Conclusion}
	\label{sec:conclusion}
	
In this paper, we established the first \emph{global} accelerated methods for non-Euclidean setting.
Our analysis demonstrated that our proposed scheme is always faster than gradient descent, and achieves the accelerated convergence after few iterations.
This establishment makes a solid progress towards understanding acceleration under non-Euclidean settings.
Our analysis builds on the potential function analysis, a recent paradigm of analyzing first-order methods.
Remarkably, our analysis for the Euclidean case \emph{exactly} recovers Nesterov's optimal method based on  estimate sequence, shedding a new light on the scope of the technique which has puzzled researchers for many years.

\setlength{\bibsep}{4pt}

\appendix

	\section{Interpretations via linear coupling}
\label{sec:linear}
Recently, \cite{allen2014linear} established a framework of designing fast first-order methods called \textit{linear coupling}.
The principle observation therein is that the two most fundamental first-order methods, namely \emph{gradient} and \emph{mirror} descent, have \emph{complementary} performances, and one can therefore design faster first-order methods by \textit{linearly coupling} the two methods.
In this section, we will discuss how one can derive from linear coupling (i) the Nesterov's optimal method iterations \eqref{nesterov} and (ii) our choice of potential function \eqref{def:poten}.

\subsection{Nesterov's iteration from linear coupling.}

One can actually derive the iteration updated rules  \eqref{nesterov} from linear coupling.
To illustrate, let us denote by  $\grad_{s\cdot \nabla }(x)$ and $\mirr_{s\cdot \nabla}(x)$ a single step of the gradient and mirror step, respectively.
If we choose the Bregman divergence associated with the mirror descent to be $D(u,v)=\frac{1}{2}\norm{u-v}_2^2$, then \eqref{nesterov} can be rewritten as follows:
\begin{align*}
    	\xg_{t+1} & \leftarrow \wg_{t+1} z_t  + (1-\wg_{t+1}) y_t  \\
    		\xm_{t+1} & \leftarrow \tilde{\wm}_{t+1} z_t  + (1-\tilde{\wm}_{t+1}) y_t  \\
    		\nabla_{t+1}&\leftarrow \nabla f(\xg_{t+1})\\
			y_{t+1} &\leftarrow  \grad_{\sg_{t+1}\nabla_{t+1}}(\xg_{t+1}) \\
			z_{t+1} &\leftarrow  \mirr_{\sm_{t+1}\nabla_{t+1}}( \xm_{t+1} ) \,,
\end{align*}
where $\tilde{\wm}_{t+1}=\wg_{t+1} +(1-\wg_{t+1})\wm_{t+1}$.
Note that these steps clearly respect linear coupling: for each step, we compute two different linear combinations of $z_t$ and $y_t$ and run gradient and mirror step from each combination to obtain the next iterates $y_{t+1}$ and $z_{t+1}$, respectively.
Indeed, the original algorithm considered in the paper~\citep{allen2014linear} chooses $\wm'\equiv 1$ and is hence a special case of the above steps.
One concrete advantage of viewing \eqref{nesterov} as above is that then \eqref{nesterov} can be naturally generalized to other settings where the smoothness of $f$ is defined with respect to a norm different from $\ell_2$.

\subsection{Choosing potential function from linear coupling.} 

Another advantage of linear coupling is that one can naturally  derive our choice of potential function \eqref{def:poten}.
To see this, we first note that the folklore analysis of gradient descent deals with the cost value $f(y)$, while that of mirror descent deals with the distance to the optimal, or more generally, the Bregman divergence $D(z,x_*)$. (See e.g. \citep[Section 2]{allen2014linear} for details.)
Since the algorithm is a \emph{linear combination} of the two methods, it is then natural to consider a \emph{linear combination} of the two performance measures, arriving at \eqref{def:poten} since  our case corresponds to the case where the Bregman divergence is chosen as $D(z,x_*)=\frac{1}{2}\norm{z-x_*}_2^2$. 

	\section{Some inequalities from Riemannian geometry (proof of Lemma~\ref{lem:betterdist})}
	\label{appen:geo}
	The main technical difficulty of analyzing optimization methods over Riemannian manifolds lies in handling its non-Euclidean metric.
	One machinery to overcome this difficulty is a classical comparison theorem due to \cite{rauch1951contribution}.
	At a high level, the theorem compares the exponential map on the manifold of interest to that on the manifold of constant sectional curvatures.

	\begin{proposition}[Rauch comparison theorem]
	\label{rauch}
		Let $M$ be a Riemannain manifold with  sectional curvatures lower bounded by $-\kappa<0$.
		Then, for any $x\in M$ and $u\in\T_xM$, the following upper bound on the operator norm of the differential of the exponential map holds:
	\begin{align*} \onorm{d(\expm_x)_{u}} \leq  \frac{\sinh(\sqrt{\kappa}\norm{u})}{\sqrt{\kappa}\norm{u}} \,.
		\end{align*}
	\end{proposition}

	\begin{proof}
		Let $u_0:=u/\norm{u}$. First, it follows from the definition that the exponential map is radially isometric, i.e.,  $\norm{d(\expm_x)_u (u_0)} =\norm{u_0}$.
		Next, due to Rauch Comparison Theorem, for any $v$ orthogonal to $u$, we have $\norm{v}\leq \norm{d(\expm_x)_u (v)} \leq \frac{\sinh(\sqrt{\kappa}\norm{u})}{\sqrt{\kappa}\norm{u}} \norm{v}$.
		Since any vector in $\T_{u}(\T_xM)$ can be represented as a linear combination of $u_0$ and vectors orthogonal to $u_0$, the proof follows.
	\end{proof}
	The above consequence of Rauch comparison theorem then implies the following metric distortion inequality, which can be seen as a global version of the ones employed in the prior arts~\cite[Lemma 9]{dyer2015riemannian} and \cite[Theorem 2]{	zhang2018estimate}.
		\begin{proposition}[Restatement of Proposition~\ref{prop:rauch}] \label{formal:rauch}
	 	Let $x,y,z$ be points on Riemannian manifold $M$ with sectional curvatures lower bounded by $-\kappa<0$.
	 	Then, for the following inequality holds:
		$$\dd{y}{z} \leq \frac{\sinh(\sqrt{\kappa}\max\{\dd{x}{y},\dd{x}{z} \})}{\sqrt{\kappa}\max\{\dd{x}{y},\dd{x}{z} \}} \cdot \DD{x}{y}{z}\,.$$
	\end{proposition}

	\begin{proof}
	To upper bound $\dd{y}{z}$ in terms of $\DD{x}{y}{z}$, consider
	a path   $p:[0,1] \to \T_x M$ defined as $p(t) =(1-t)\cdot \lm{x}{y} + t\cdot \lm{x}{z}$.
	Then, its image  $\expm_x(p)$ is a path on $M$ connecting $y$ from $z$.
	By definition of the distance on the manifold, $\dd{y}{z}$ is clearly upper bounded by the length of $\expm_x(p)$.
	On the other hand, using Proposition~\ref{rauch}, the length of $\expm_x(p)$ can be upper bounded as follows (since $\norm{p'(t)} =\norm{\lm{x}{y}-\lm{x}{z}} = \DD{x}{y}{z}$):	
	\begin{align*}
	\int_{0}^1 \norm{\frac{d}{dt}\expm_x\big( p(t)\big)} dt &\leq \int_{0}^1 \onorm{d(\expm_x)_{p(t)}}\cdot \norm{p'(t)} dt \\
		&\leq S_\kappa(\max\{\dd{x}{y},\dd{x}{z} \}) \cdot \DD{x}{y}{z}\,,
		\end{align*}
		where the last inequality follows from  the fact that $\norm{p(t)} $ is upper bounded by $\max\{\norm{p(0)},\norm{p(1)}\}=\max\{\dd {x}{y},\dd{x}{z}\}$.
	\end{proof} 
	Next machinery is a user-friendly global trigonometric inequality from~\cite[Lemma 6]{zhang2016first} and \cite[Lemma 3.12]{cordero2001riemannian}:
	
	\begin{proposition}[Global trigonometric inequality]
		\label{lem:trig}
		Let $M$ be a Riemannain manifold with  sectional curvatures lower bounded by $-\kappa<0$.
		Let $x,y,z$ be the vertices of a geodesic triangle with the lengths of the opposite side being $a,b,c$, respectively,
		and $A$ be the angle of the triangle at the vertex $x$, then we have the following inequality:
		$$a^2 \leq \frac{\sqrt{\kappa}c}{\tanh(\sqrt{\kappa}c)} \cdot b^2 + c^2 -2bc \cos A\,.$$ 
	\end{proposition}
\begin{proof}
See \citep[Section 3.1]{zhang2016first}.
\end{proof}
Now based on these two machineries, we prove the improved metric distortion inequality (Lemma~\ref{lem:betterdist}):
\begin{lembox}[Formal statement of Lemma~\ref{lem:betterdist}]  \label{dist:formal}
		Let $x,y,z$ be points on Riemannian manifold $M$ with sectional curvatures lower bounded by $-\kappa<0$.
		For function $\widehat{\vd}:\re_{\geq 0}\to \re_{\geq 1}$ defined as
	\begin{align*}
		\widehat{\vd}(r) := \begin{cases}
		\min_{\eps>0}\max\left\{ 1+ \left(1+\eps^{-1}\right)^{2}\left(\frac{\sqrt{\kappa}r}{\tanh(\sqrt{\kappa}r)} -1\right),~\left(\frac{\sinh\left((1+\eps)\sqrt{\kappa}\cdot r\right)}{(1+\eps)\sqrt{\kappa}\cdot r}\right)^2 \right\} & \text{if $r>0$},\\
		1, &\text{if $r=0$,}
		\end{cases} 
	\end{align*}
	the following inequality holds: 
		$\dd{y}{z}^2\leq \widehat{\vd}(\dd{x}{y})\cdot \DD{x}{y}{z}^2$.
\end{lembox}
 
 	\begin{remark}[Properties of $\widehat{\vd}$]\label{rmk:express}
	Note that $\widehat{\vd}\leq \vd$ since $\vd$ corresponds to choosing $\eps=1$ for all $r$ in $\widehat{\vd}$. 
	Hence Lemma~\ref{dist:formal} implies Lemma~\ref{lem:betterdist}.
		In particular, the followings are true about $\widehat{\vd}(r)$:
		\begin{enumerate}
		    \item $\lim_{r\to 0+}\widehat{\vd}(r) = 1$.
		    \item For any $\eps>0$, there exists $R_{\eps}>0$ such that  $\widehat{\vd}(r) \leq\left(\frac{\sinh\left((1+\eps)\sqrt{\kappa}\cdot r\right)}{(1+\eps)\sqrt{\kappa}\cdot r}\right)^2 $ for $r>R_{\eps}$.
		\end{enumerate}
  Hence, one can essentially regard $\widehat{\vd}$ as a \emph{proxy} for $S_{\kappa}$.
	\end{remark}

	\begin{proof}[Proof of Lemma~\ref{dist:formal}]
	Let us fix an arbitrary constant $\eps>0$. 
	We will separately handle two cases: (i) $ (1+\eps)\cdot\dd{x }{y} < \dd{x}{z}$ and (ii) $ (1+\eps)\cdot \dd{x}{y} \geq    \dd{x }{z}$. 
	Let us begin with the first case.
Applying the trigonometric lemma to $\triangle xyz$,  and letting $\tri{} := \frac{\sqrt{\kappa}\dd{x}{y}}{\tanh(\sqrt{\kappa}\dd{x}{y})}$, we obtain:
		\begin{align*}
		\dd{y}{z}^2 &\leq \dd{x}{y}^2 + \tri{}\cdot \dd{x}{z}^2 -2\inp{\lm{x}{y}}{\lm{x}{z}}\\
	&=(\tri{}-1)\cdot \dd{x}{z}^2 + \dd{x}{y}^2 +  \dd{x}{z}^2 -2\inp{\lm{x}{y}}{\lm{x}{z}}\\
	&=(\tri{}-1)\cdot\dd{x}{z}^2 + \DD{x}{y}{z}^2\,.
		\end{align*}
		where the last line follows from the Euclidean law of cosine.
	On the other hand, from the Euclidean triangle inequality (consider the triangle $\triangle xyz$ in the tangent space $\T_{x}M$), $\DD{x}{y}{z} \geq (\dd{x }{z}-\dd{x}{y})>  \frac{\eps}{1+\eps} \cdot \dd{x}{z}$.
		Hence, combining these two, we get
		\begin{align}
	  \dd{y}{z}^2 &\leq (\tri{}-1)\cdot \dd{x}{z}^2 + \DD{x}{y}{z}^2 \nonumber\\
		&\leq \left(1+\eps^{-1}\right)^{2}\cdot \left(\tri{}-1\right) \cdot \DD{x}{y}{z}^2  + \DD{x }{y}{z}^2 \nonumber\\
		&= \left[1+\left(1+\eps^{-1}\right)^{2}\cdot  \left(\tri{}-1\right) \right] \cdot \DD{x}{y}{z}^2\,. \label{ineq:1}
		\end{align}
		Next, for the case  $(1+\eps)\cdot \dd{x}{y} \geq   \dd{x }{z}$,  Proposition~\ref{formal:rauch} implies:
		\begin{align}
		\dd{y}{z}^2 &\leq \left(\frac{\sinh\left((1+\eps)\sqrt{\kappa}\cdot \dd{x}{y} \right)}{(1+\eps)\sqrt{\kappa}\cdot \dd{x}{y}}\right)^2\cdot \DD{x}{y}{z}^2\,. \label{ineq:2}
		\end{align}
		Therefore, combining \eqref{ineq:1} and \eqref{ineq:2}, the proof is completed.
	\end{proof}

	\section{Proofs/calculations missing from the main text}
	\subsection{Proofs of folklore analyses of gradient steps (Propositions \ref{folk:grad} and \ref{folk :mirror})}
	\label{pf:folk}
 
 {\bf Proof of Proposition~\ref{folk:grad}:}
 By the $L$-smoothness of $f$, we have
$f(y) \leq f(x) +\inp{\nabla f(x)}{y-x} + \frac{L}{2}\norm{x-y}^2 = f(x) +\inp{\nabla f(x)}{-s\nabla f(x)} + \frac{L}{2}\norm{-s\nabla f(x)}^2 = f(x) -s\left(1 -\frac{Ls}{2} \right)\norm{\nabla f(x)}^2$.  \hfill\ensuremath{\blacksquare}\\
  {\bf Proof of Proposition~\ref{folk :mirror}:}
 The proof is an elementary calculation: $\norm{z-x_*}^2 = \norm{z-x+x-x_*}^2=\norm{z-x}^2 +\norm{x-x_*}^2 +2\inp{z-x}{x-x_*}= s^2\norm{\nabla f(x)}^2 +\norm{x-x_*}^2 +2s\inp{\nabla f(x)}{x_*-x}$.   \hfill\ensuremath{\blacksquare}

	\subsection{Derivation of the upper bound on the potential difference \eqref{upper:1}}
	\label{app:derive}
	First, one can easily express \eqref{eq:md} in terms of $\nabla,X,W$ using Proposition~\ref{folk :mirror}:
	\begin{align}
	    \eqref{eq:md}&=  B_{t+1}\cdot  \norm{z_{t+1}-x_*}^2-B_t\cdot \norm{z_t-x_*}^2 \nonumber\\
	    &= B_{t+1}\cdot \norm{X+\wm W}^2 + B_{t+1}\sm^2 \cdot \norm{\nabla}^2 -2B_{t+1}\sm\cdot \inp{\nabla}{X+\wm W} -B_t \cdot \norm{W+X}^2 \nonumber\\
\begin{split}
&= (B_{t+1}-B_t) \cdot \norm{X}^2
+ (\wm^2 B_{t+1}-B_t)\cdot \norm{W}^2 + \sm^2B_{t+1}\cdot \norm{\nabla}^2\\
&\quad+2(\wm B_{t+1}-B_t)\cdot \inp{X}{W}-2\wm\sm B_{t+1} \inp{W}{\nabla}-2\sm B_{t+1}\cdot \inp{X}{\nabla}\,.
\end{split} \label{md:final}
\end{align}
For  \eqref{eq:gd}, we apply Propostion~\ref{folk:grad} and rearrange terms to obtain:
\begin{align}
    \eqref{eq:gd}&\leq A_{t+1} \cdot \left(f(x_{t+1})-f(x_*)\right)-\drop A_{t+1}\cdot \norm{\nabla}^2 -A_t\cdot \left( f(y_t)-f(x_*) \right) \nonumber\\
    &=A_t  \cdot \left(f(x_{t+1})-f(y_t)\right)+(A_{t+1}-A_t)\cdot \left( f(x_{t+1})-f(x_*)\right)-\drop A_{t+1}\cdot \norm{\nabla}^2\,. \nonumber
\end{align}
Now using the inequality $f(u)-f(v)\leq \inp{\nabla f(u)}{u-v} -\frac{\mu}{2}\norm{u-v}^2$ for all $u,v$ ($\because$  $\mu$-strong convexity of $f$) together with the identity $x_{t+1}-y_t =\frac{\wg}{1-\wg} (z_t-x_{t+1}) = \frac{\wg}{1-\wg} W$, 
one can also derive an upper bound on \eqref{eq:gd} in terms of $\nabla,W,X$: 
\begin{align}
 \eqref{eq:gd} &\leq \frac{\wg}{1-\wg}A_t \cdot \inp{\nabla}{W} - \frac{\mu}{2}\left(\frac{\wg}{1-\wg}\right)^2 A_t\cdot  \norm{W}^2  \nonumber\\
 &\quad+(A_{t+1}-A_t)\cdot \inp{\nabla}{X}-\frac{\mu}{2}(A_{t+1}-A_t)\cdot \norm{X}^2 -\drop A_{t+1}\cdot \norm{\nabla}^2 \nonumber\\
 \begin{split}
 &= -  \frac{\mu}{2}\left(\frac{\wg}{1-\wg}\right)^2 \cdot A_t\cdot  \norm{W}^2 -\frac{\mu}{2}(A_{t+1}-A_t)\cdot \norm{X}^2-\drop A_{t+1}\cdot \norm{\nabla}^2\\
 &\quad+ \frac{\wg}{1-\wg}A_t \cdot \inp{W}{\nabla} +(A_{t+1}-A_t)\cdot \inp{X}{\nabla}\,.
 \end{split} \label{gd:final}
\end{align}
Putting \eqref{md:final} and \eqref{gd:final} together, we obtain \eqref{upper:1}.

\subsection{Proof of Theorem~\ref{thm:euclidean}}
\label{pf:thm1}
 From the equality version of \eqref{b:from:a}, i.e.,
$B_{t+1} = (A_{t+1}-A_t)^2/(4\drop \cdot  A_{t+1})$.
one can easily express $B_{t+1}$ in terms of $\cc_{t+1}$.
More specifically, using the relation  $ 1-\cc_{t+1}:=A_{t}/A_{t+1}$, we have $ (\frac{A_{t+1}-A_t}{A_{t+1}})^2= \cc_{t+1}^2$, and hence:
\begin{align}
       B_{t+1} = \left(\frac{A_{t+1}-A_t}{A_{t+1}}\right)^2\cdot \frac{A_{t+1}}{4\drop}  =\frac{\cc_{t+1}^2}{1-\cc_{t+1}} \cdot \frac{A_t}{4\drop } \,. \label{eq:b}
\end{align}
From this, one can also conclude that 
\begin{align}
    \frac{A_{t+1}}{B_{t+1}} = \frac{A_t}{(1-\cc_{t+1})}\cdot \frac{1-\cc_{t+1}}{\cc_{t+1}^2}\cdot \frac{4\drop}{A_t} = \frac{4\drop}{\cc_{t+1}^2}\,. \label{eq:ratio}
\end{align}
Now from \eqref{stepsizes}, \eqref{eq:recur}, \eqref{eq:b} and \eqref{eq:ratio}, one can express $\wg_{t+1},\wm_{t+1},\sm_{t+1}$ in terms of $\cc_{t+1}$:
\begin{align*}
    \sm_{t+1} &\overset{\eqref{stepsizes}}{=} \frac{A_{t+1}-A_t}{2B_{t+1}} = \frac{A_{t+1}-A_t}{A_{t+1}}\cdot \frac{A_{t+1}}{2B_{t+1}} \overset{\eqref{stepsizes}\&\eqref{eq:ratio}}{=} \cc_{t+1}\cdot \frac{2\drop}{\cc_{t+1}^2} =2\drop\cc_{t+1}^{-1}\,,\\
    \wm_{t+1} &\overset{\eqref{stepsizes}}{=} \frac{B_t}{B_{t+1}} \overset{\eqref{eq:b}}{=} \frac{1-\cc_{t+1}}{\cc_{t+1}^2}\cdot 4\drop \cdot \frac{B_t}{A_t}\overset{\eqref{eq:recur}}{=} \frac{\cc_{t+1}-2\mu\drop}{\cc_{t+1}} = 1-2\mu\drop\cc_{t+1}^{-1}\,,\quad \text{and}\\
    \frac{\wg_{t+1}}{1-\wg_{t+1}}&\overset{\eqref{stepsizes}}{=}\frac{(A_{t+1}-A_t)B_t}{A_t B_{t+1}} = \frac{A_{t+1}-A_t}{A_{t+1}}\cdot \frac{B_t}{A_t}\cdot \frac{A_{t+1}}{B_{t+1}}\\
    &\overset{\eqref{eq:recur}\&\eqref{eq:ratio}}{=} \cc_{t+1}\cdot \frac{\cc_{t+1}(\cc_{t+1}-2\mu\drop) }{4\drop (1-\cc_{t+1})} \cdot \frac{4\drop}{\cc_{t+1}^2}=\frac{\cc_{t+1}- 2\mu\drop }{1-\cc_{t+1}}\,.
\end{align*}
With the above choices of parameters, one can easily check that $\wg_{t+1},\wm_{t+1}$ both lie in $[0,1]$ since $\cc_{t+1}\in [2\mu\drop ,1)$. In particular,  $\wm_{t+1}^2 B_{t+1}\leq \wm_{t+1} B_{t+1} = B_t$, implying $C_1\leq 0$.
Therefore, the above choices of parameters satisfy $C_1,C_2,C_3\leq 0$ and $C_4,C_5,C_6=0$, and consequently, $\Phi_{t+1}\leq \Phi_t$ since $\Phi_{t+1}-\Phi_t\leq \eqref{upper:1}$.
This completes the proof of Theorem~\ref{thm:euclidean}.

\subsection{Analysis of recursive relations (\eqref{recur:xi} and \eqref{r:recur:xi})}
\label{pf:conv}
For simplicity, we replace $2\mu\drop$ with a constant $a\in(0,1)$ and consider:
\begin{align}
\frac{\cc_{t+1}(\cc_{t+1}-a) }{1-\cc_{t+1}} = \frac{1}{\ddr{} }\cdot \cc_{t}^2\,. \label{appen:recur}
\end{align}
In particular, $\ddr{}=1$ and $a=2\mu\drop$ recovers \eqref{recur:xi}.
Below, we state and prove a formal statement of Lemma~\ref{lem:conv} which demonstrates the geometric convergence of the $\cc_t$ to the fixed point.
\begin{lembox}  \label{lem:appen}
For any constants $\ddr{}\geq 1$, $a\in(0,1)$, and
 an  initial value $\cc_0\geq 0$, 
Then the followings properties are true about \eqref{appen:recur}: 
\begin{enumerate}
\item $\cc(\ddr{}):=(\sqrt{(\ddr{}-1)^2+4\ddr{}a }-(\ddr{}-1))/2$ is the unique fixed point of \eqref{appen:recur}. 
    \item $\lim_{t\to \infty}\cc_t \searrow \cc(\ddr{})$  if $\cc_0>\cc(\ddr{})$; $\cc_t \equiv \cc(\ddr{})$ if $\cc_0=\cc(\ddr{})$; and $\lim_{t\to \infty}\cc_t \nearrow \sqrt{2\mu\drop}$ if $0\leq \cc_0<\cc(\ddr{})$.
    \item $|\cc_t -\cc(\ddr{})| \leq \left(\frac{1}{\sqrt{\ddr{}}} \left(1-\frac{4}{5+\sqrt{5}}\cdot \frac{a}{\sqrt{\ddr{}}}\right)\right)^t |\cc_0 -\cc(\ddr{})|$, i.e., the convergences are geometrical.
\end{enumerate}
\end{lembox}
\begin{proof} 
Define $\cho(v):=\frac{v(v-a)}{1-v}$ and $\cht(v):=\frac{1}{\ddr{}}v^2$.
Then, using these notations, \eqref{appen:recur} becomes
\begin{align} \label{appen:equiv}
\cho(\cc_{t+1})=\cht(\cc_t)\,.
\end{align}
Now, in order to understand \eqref{appen:equiv}, let us study the properties of the two functions.
First, it is straightforward to check  that $\cht$ is increasing on $\re_{\geq 0}$ and $\cho$ is increasing on $[a,1)$ with $\cho(a)=0$ and $\lim_{v\to 1-} \cho(v) = \infty$. 
Indeed,  $\cho$ is increasing since  $\frac{d}{dv}\cho(v)=\frac{1-a}{(1-v)^2}-1 \geq \frac{1}{1-a} -1 >0$.

Hence, one can consider the inverse of the restriction $\cho|_{[a,1)}$.
We will write the inverse as $\cho^{-1}$.
Letting $\ctt:=\cho^{-1}\circ \cht$, \eqref{appen:equiv} can be written as the standard recursive relation form
\begin{align}
    \cc_{t+1} = \ctt (\cc_{t})\,. \label{appen:stan}
\end{align}
Note that $\ctt:\re_{\geq 0} \to [a,1)$, and hence, $\cc_t\in (a,1]$ for $t\geq 1$. 
Solving $\cho(v)=\cht(v)$ on $v\in[a,1)$ yields $v=\cc(\ddr{})$.
Hence, $\cc(\ddr{})$ is the unique fixed point of \eqref{appen:stan}.
From this, together with the fact that both $\cho$ and $\cht$ are increasing,  we have $\cho<\cht$  for $x\in [a,\cc(\ddr{}))$, and $\cho>\cht$  for $x\in (\cc(\ddr{}),1)$.
Hence, $\{\cc_{t}\}$ is increasing if $\cc_0 \in [0,\cc(\ddr{}))$ and decreasing if $\cc_0>\cc(\ddr{})$.

Now we prove the geometric convergence of \eqref{appen:stan} to $\cc(\ddr{})$.
To that end, let us first compute $\ctt$:
  	    \begin{align*}
	        \ctt(v)= \cho^{-1}(\cht(v)) = \cho^{-1}(v^2/\ddr{}) = \frac{1}{2}\left( \sqrt{(v^2/\ddr{}-a)^2 +4v^2/\ddr{}}-(v^2/\ddr{}-a) \right)\,.
	    \end{align*}
By mean value theorem, the key to showing geometric convergences of \eqref{appen:stan} is to bound the derivative of $\ctt$.
In particular, if we can establish that $|\ctt'(v) | \leq K <1$ for $v\in[a,1)$, then we have  
\begin{align} \label{appen:mvt}
    \left|\cc_{t+1}-\cc(\ddr{})\right| = \left|\ctt(\cc_t)- \ctt(\cc(\ddr{})) \right| \leq K\cdot \left| \cc_t -\cc(\ddr{}) \right|\,.
\end{align}
Defining $\cde(v) =\frac{v(v^2-a)+2v}{\sqrt{(v^2-a)^2 +4v^2}} - v$,	one can express $\ctt$ in terms of $\cde$:
\begin{align*}
	        \ctt'(v) &= \frac{\frac{v}{\ddr{}}(v^2/\ddr{}-a)+2\frac{v}{\ddr{}}}{\sqrt{(v^2/\ddr{}-a)^2 +4v^2/\ddr{}}} - \frac{v}{\ddr{}} = \frac{1}{\sqrt{\ddr{}}}\cdot \left( \frac{\frac{v}{\sqrt{\ddr{}}}(v^2/\ddr{}-a)+2\frac{v}{\sqrt{\ddr{}}}}{\sqrt{(v^2/\ddr{}-a)^2 +4v^2/\ddr{}}} - \frac{v}{\sqrt{\ddr{}}}\right)\\
	        &= \frac{1}{\sqrt{\ddr{}}}\cdot \cde(v/\sqrt{\ddr{}})
	    \end{align*}
 Hence, it suffices to show that $\cde(v)<1$ for $v\in (0,1)$.
 \begin{proposition}\label{prop:der}
  The following inequality holds for $v\in(0,1)$: $0\leq\cde(v) <1-\frac{4}{5+\sqrt{5}}\cdot v$.
\end{proposition}
\begin{proof}
$\cde(v)\geq 0$ trivially holds since, $\ctt$ is a composition of increasing functions and hence increasing.
Now let us prove the upper bound.
We first consider the case $a<v\leq\sqrt{a}$. 
Since $v^2\leq a$, we have
\begin{align*}
    \cde(v)&=\frac{-v(a-v^2)+2v}{\sqrt{(v^2-a)^2 +4v^2}} - v \leq \frac{ 2v}{\sqrt{(v^2-a)^2 +4v^2}} - v \leq 1-v\,.
\end{align*}
Next, consider the case $v>\sqrt{a}$. Then, $v^2>a$, and hence
\begin{align*}
    \cde(v)&=\frac{v(v^2-a)+2v}{\sqrt{(v^2-a)^2 +4v^2}} - v = \frac{ 2v}{\sqrt{(v^2-a)^2 +4v^2}} - v \cdot  \frac{\sqrt{(v^2-a)^2 +4v^2}-(v^2-a)}{\sqrt{(v^2-a)^2 +4v^2}}\\
    &= \frac{ 2v}{\sqrt{(v^2-a)^2 +4v^2}} - v \cdot  \frac{4v^2}{ (v^2-a)^2 +4v^2 +(v^2-a)\sqrt{(v^2-a)^2 +4v^2}}\\
    &\overset{(\clubsuit)}{\leq} 1 - v \cdot  \frac{4v^2}{ v^2 +4v^2 +v\sqrt{v^2 +4v^2}} = 1- \frac{4 }{5+\sqrt{5}}\cdot v\,.
\end{align*}
where ({\footnotesize $\clubsuit$}) follows since $v\in (\sqrt{a},1)$; in particular, it follows that $0\leq v^2-a \leq v^2 \leq v$.
Combining the two cases, we complete the proof.
\end{proof}
 From Proposition~\ref{prop:der} and \eqref{appen:mvt}, the geometric convergence follows.
 \end{proof}
 
	 \subsection{Proof of Theorem~\ref{thm:riem}}
	 \label{pf:thm2}
	 
	 We first introduce the definitions of geodesical smoothness and (strong) convexity for completeness.
	For simplicity, we assume that the function $f:M\to \re$ is differentiable throughout the definitions, and following the notation for the Euclidean case, we will denote by $\nabla f(x)$ the gradient of $f$ at $x$.
 
	 	 \begin{definition}[Geodesical (strong) convexity]   \label{def:gconvex}$f$ is said to be geodesically $\mu$-strongly convex if 
$$ f(y)\geq f(x) +\inp{\nabla f(x) }{\lm{x}{y}}_x + \frac{\mu}{2}\dd{x}{y}^2\quad \text{for any $x, y \in M$,}$$
where $\inp{\cdot}{\cdot }_x$ denotes the inner product in the tangent space of $x$ induced by the Riemannian metric. 
	 \end{definition}
	 
	 \begin{definition}[Geodesical smoothness]
	  $f: M\to \re$ is said to be geodesically $L$-smooth if  
\begin{align*}
    f(y) \leq f(x) + \inp{\nabla f(x)}{\lm{x}{y}}_x +\frac{L}{2}\dd{x}{y}^2\quad\text{for any $x,y\in M$.}
\end{align*}
An Equivalent definition is 
\begin{align*}
    \norm{\nabla f(x) - \Gamma_y^x \nabla f(y)}_x \leq L\cdot \dd{x}{y}\quad\text{for any $x,y\in M$,}
\end{align*}
where $\Gamma_y^x$ is the parallel transport from $y$ to $x$.
	 \end{definition} 
	 Now with these definitions, one can establish analogues of Propositions \ref{folk:grad} and \ref{folk :mirror}:
	 \begin{proposition} \label{r:folk:grad}
	 Let $y =\exm{x}{-s\cdot \nabla f(x)}$. If $f$ is geodesically $L$-smooth, then 
$f(y)-f(x) \leq -s\left( 1 - Ls/2 \right)  \norm{\nabla f(x) }_x^2$.
 \end{proposition} 
 \begin{proof}
 Identical to that of Proposition~\ref{folk:grad}.
 \end{proof}
\begin{proposition} \label{r:folk :mirror}
  Let $z=  \exm{u}{v-s\cdot\nabla f(u)}$ for some vector $v\in T_x M$.
   Then, for any $x_*$, 
   $\DD{u}{z}{x_*}^2- \DD{u}{\exm{u}{v}}{x_*}^2 = s^2 \norm{\nabla f(u)}_u^2 + 2s \inp{\nabla f(u)}{\lm{u}{x_*}-v}_u$.
 \end{proposition}
 {\bf Proof} 
 The proof follows immediately from the definition of the projected distances (Definition~\ref{def:pd}):
 \begin{align*}
     \DD{u}{z}{x_*}^2 &= \norm{\lm{u}{z} -\lm{u}{x_*}}_u^2\\
     &=\norm{v-s\cdot \nabla f(u) -\lm{u}{x_*}}_u^2\\
     &= \norm{v -\lm{u}{x_*}}_u^2 +\norm{-s\cdot \nabla f(u)}_u^2 +2\inp{-s\cdot \nabla f(u)}{v -\lm{u}{x_*}}_u\\
     &= \DD{u}{\exm{u}{v}}{x_*}^2 +s^2 \norm{\nabla f(u)}_u^2 +2\inp{\cdot \nabla f(u)}{\lm{u}{x_*}-v}_u\,. &{\blacksquare}
 \end{align*}	 
 
 Now we prove Theorem~\ref{thm:riem}.
It turns out one can establish an upper bound of the potential difference  $\Psi_{t+1}-\Psi_t$ similar to \eqref{upper:1}.
Here the difference is that instead of $W,X,\nabla$, we now have the following three vectors in the same tangent space $T_xM$:
\begin{align}
    \WW:= \lm{x_{t+1}}{z_t}\,,\quad \XX:= -\lm{x_{t+1}}{x_*}\,,~~ \text{and}~~ \NN:=\grad f(x_{t+1})\,.
\end{align}
As pointed out in \S\ref{sec:poten}, the fact that the three vectors are in the same tangent space (since they are all rooted at $x_{t+1}$) is crucial for the analysis to follow.
With these vectors and Propositions~\ref{r:folk:grad} and \ref{r:folk :mirror}, one can derive the following upper bound similarly to Appendix~\ref{app:derive} (here $\norm{\cdot}$ denotes $\norm{\cdot}_{x_{t+1}}$):
\begin{mdframed}
$\CC_1\cdot \normp{\WW}^2 + \CC_2\cdot \normp{\XX}^2 +\CC_3\normp{\NN}^2  +\CC_4 \cdot \inpp{\WW}{\XX} +\CC_5\cdot \inpp{\WW}{\NN}+\CC_6\cdot \inpp{\XX}{\NN}\,,\eqnum \label{r:upper}$\\
where 
$\begin{cases}
  \CC_1:=\wm_{t+1}^2 B_{t+1}-\frac{B_t}{\pmb{\ddr{t+1}}}  -  \frac{\mu}{2}\frac{\wg_{t+1}^2}{(1-\wg_{t+1})^2}  A_t\,, &    \CC_2:=B_{t+1}-\frac{B_t}{\pmb{\ddr{t+1}}} - \frac{\mu}{2}(A_{t+1}-A_t)\,,\\
    \CC_3:= \sm_{t+1}^2 B_{t+1} - \drop \cdot A_{t+1}\,, &
    \CC_4:=2 \cdot \left(\wm_{t+1} B_{t+1} -\frac{B_t}{\pmb{\ddr{t+1}}}  \right)\,,\\
    \CC_5:= \frac{1-\wg}{\wg} A_t-2\wm_{t+1}\sm_{t+1} B_{t+1}\,,\quad\text{and} &
    \CC_6:=(A_{t+1}-A_t)-2\sm_{t+1} B_{t+1}\,.
    \end{cases}$
\end{mdframed}
Notice the similarity between \eqref{r:upper} and \eqref{upper:1}: the only difference  is that $B_t$'s in \eqref{upper:1} are replaced with $B_t/\ddr{t+1}$'s.
Indeed, this difference is precisely attributed to the definition of valid distortion rate (Definition~\ref{def:val}).
More specifically, in the derivation of \eqref{r:upper}, we use  $-B_t\cdot \DD{x_{t}}{z_t}{x_*}^2 \leq -\frac{B_t}{\ddr{t}} \cdot  \DD{x_{t+1}}{z_t}{x_*}^2$, which precisely accounts for the appearance of $B_t/\ddr{t+1}$ instead of $B_t$:

	We now follow   Section~\ref{sec:ensure} to make \eqref{r:upper} ``$-$SoS.''
First, from $\CC_4=\CC_5=\CC_6=0$, we get:
\begin{align}
    \sm_{t+1}  = \frac{A_{t+1}-A_t}{2B_{t+1}}\,,\quad \wm_{t+1} = \frac{B_t}{\pmb{\ddr{t+1}} B_{t+1}}\,,~\text{and}~\frac{\wg_{t+1}}{1-\alpha_{t+1}} = \frac{ (A_{t+1}-A_t)B_t}{\pmb{\ddr{t+1}}A_t B_{t+1}}\,. 
    \label{r:stepsizes}
\end{align}
Moving on, one can similarly obtain the following from $\CC_3\leq 0$ and $\CC_2\leq 0$:
\begin{align}
    &\frac{(A_{t+1}-A_t)^2}{4\drop \cdot  A_{t+1}} \leq B_{t+1} \quad \text{and}  \label{r:b:from:a}\\
    &\frac{(A_{t+1}-A_t)^2}{4\drop\cdot  A_{t+1}} -  (A_{t+1}-A_t)\frac{\mu}{2} \leq \frac{B_t}{\pmb{\ddr{t+1}}}\,.     \label{r:ineq:a}
\end{align}
Again, using the suboptimality shrinking ratio $ 1-\cc_{t+1}:=A_{t}/A_{t+1}$, 
\eqref{r:ineq:a} becomes
\begin{align}
      \frac{\cc_{t+1}(\cc_{t+1}-2\mu\drop)}{1-\cc_{t+1}} \leq \frac{4\drop}{\pmb{\ddr{t+1}}}\cdot  \frac{B_t}{A_t}\,. \label{r:ineq:a2}
\end{align}
 Again, due to  the LHS being an increasing function on $\cc_{t+1}\in[2\mu\drop,1)$,  the largest $\cc_{t+1}$ (or equivalently, the largest $A_{t+1}$) satisfies \eqref{r:ineq:a2} (or equivalently, \eqref{r:ineq:a})  with equality:
 \begin{align}
      \frac{\cc_{t+1}(\cc_{t+1}-2\mu\drop)}{1-\cc_{t+1}} = \frac{4\drop}{\pmb{\ddr{t+1}}}\cdot  \frac{B_t}{A_t}\,. \label{r:eq:appen}
\end{align}
 Consequently, such choice of $\cc_{t+1}$ (or corresponding $A_{t+1}$) also satisfies \eqref{r:b:from:a} with equality.

Now one can follow the calculations in Appendix~\ref{pf:thm1} to express parameters in terms of $\cc_{t+1}$.
From the equality version of \eqref{r:b:from:a},
one can easily deduce the identical relations as \eqref{eq:b} and \eqref{eq:ratio}.
Now using \eqref{eq:b}, \eqref{eq:ratio}, \eqref{r:stepsizes}, and \eqref{r:eq:appen}, one can express $\wg_{t+1},\wm_{t+1},\sm_{t+1}$ in terms of $\cc_{t+1}$.
It turns out that  the final expressions do not depend on $\ddr{t+1}$ and are \emph{identical} to Appendix~\ref{pf:thm1}:
\begin{align*}
    \sm_{t+1} &\overset{\eqref{r:stepsizes}}{=} \frac{A_{t+1}-A_t}{2B_{t+1}} = \frac{A_{t+1}-A_t}{A_{t+1}}\cdot \frac{A_{t+1}}{2B_{t+1}} \overset{\eqref{r:stepsizes}\&\eqref{eq:ratio}}{=} \cc_{t+1}\cdot \frac{2\drop}{\cc_{t+1}^2} =2\drop\cc_{t+1}^{-1}\,,\\
    \wm_{t+1} &\overset{\eqref{r:stepsizes}}{=} \frac{B_t}{\pmb{\ddr{t+1}}B_{t+1}} \overset{\eqref{eq:b}}{=} \frac{1-\cc_{t+1}}{\cc_{t+1}^2}\cdot \frac{4\drop}{\pmb{\ddr{t+1}}} \cdot \frac{B_t}{A_t}\overset{\eqref{r:eq:appen}}{=} \frac{\cc_{t+1}-2\mu\drop}{\cc_{t+1}} = 1-2\mu\drop\cc_{t+1}^{-1}\,,\quad \text{and}\\
    \frac{\wg_{t+1}}{1-\wg_{t+1}}&\overset{\eqref{r:stepsizes}}{=}\frac{(A_{t+1}-A_t)B_t}{\ddr{t+1}A_t B_{t+1}} = \frac{A_{t+1}-A_t}{A_{t+1}}\cdot \frac{B_t}{\ddr{t+1}A_t}\cdot \frac{A_{t+1}}{B_{t+1}}\\
    &\overset{\eqref{r:eq:appen}\&\eqref{eq:ratio}}{=} \cc_{t+1}\cdot \frac{ \cc_{t+1}(\cc_{t+1}-2\mu\drop) }{4\drop (1-\cc_{t+1})} \cdot \frac{4\drop}{\cc_{t+1}^2}= \frac{\cc_{t+1}- 2\mu\drop }{1-\cc_{t+1}}\,.
\end{align*}

With the above choices of parameters, one can again check that $\wg_{t+1},\wm_{t+1}$ both lie in $[0,1]$ since $\cc_{t+1}\in [2\mu\drop ,1)$. In particular,  $\wm_{t+1}^2 B_{t+1}\leq \wm_{t+1} B_{t+1} =  B_t/\pmb{\ddr{t+1}}$, implying $\CC_1\leq 0$.
Therefore, the above choices of parameters satisfy $\CC_1,\CC_2,\CC_3\leq 0$ and $\CC_4,\CC_5,\CC_6=0$, and consequently, $\Psi_{t+1}\leq \Psi_t$ since $\Psi_{t+1}-\Psi_t\leq \eqref{r:upper}$.
This completes the proof of Theorem~\ref{thm:riem}.

	 \subsection{Proofs of distance shrinking results (Lemma~\ref{lem:shrink:d})}
\label{pf:shrink}

We first show the following result which is a direct consequence of Theorem~\ref{thm:riem}:
	\begin{proposition} 
	\label{prop:shrink}
	Assume that $\mu>0$ and $\sg \in(0,2/L)$.
	Let $\init :=  f(x_0)-f(x_*) + \frac{1}{4\drop} \cc_{0}^2\cdot \dd{x_0}{x_*}^2$.
Then, $x_t$, $y_t$, $z_t$ generated from Algorithm~\ref{alg:1} satisfy the following distance bounds:
	\begin{enumerate}
	    \item 
	   $\DD{x_{t}}{z_{t}}{x_*} \leq \sqrt{\init\prod_{j=1}^{t}(1- \cc_j)} \cdot \sqrt{\frac{1}{\mu^2\drop} }$.
	    \item $\dd{y_{t}}{x_*} \leq \sqrt{\init \prod_{j=1}^{t}(1- \cc_j)}\cdot \sqrt{\frac{2}{\mu}} $.
	   \item $\DD{x_t}{y_{t}}{z_{t}} \leq  \sqrt{ \init \prod_{j=1}^{t}(1- \cc_j)} \cdot \left(\sqrt{\frac{2}{\mu}}+\sqrt{\frac{1}{\mu^2\drop}}\right)$.
	 \end{enumerate}
	\end{proposition}
	\begin{proof}
	    By recursively applying Theorem~\ref{thm:riem}, we have the following for any $t\geq 1$:
	  \begin{align*}
	      f(y_{t})-f(x_*) + \frac{1}{4\drop} \cc_{t}^2\cdot \DD{x_{t}}{z_{t}}{x_*}^2&\leq \prod_{j=1}^{t} (1-\cc_{j}) \cdot \left[ f(y_0)-f(x_*) + \frac{1}{4\drop} \cc_{0}^2\cdot \DD{x_0}{z_0}{x_*}^2 \right]\\
	     &=\prod_{j=1}^{t} (1-\cc_{j}) \cdot \init \,,
	  \end{align*}   
	  where the equality follows since $x_0=y_0=z_0$ (which implies $\DD{x_0}{z_0}{x_*} = \dd{x_0}{x_*}$).

	   Hence, the bound on $\DD{x_t}{z_t}{x_*}$ follows immediately due to $\xi_t \in [2\mu\drop,1)$.
	     As for the bound on $\dd{y_t}{x_*}$, 
	   it follows from the $\mu$-strong g-convexity of $f$ (Definition~\ref{def:gconvex}), which implies $\frac{\mu}{2}\cdot \dd{y_{t}}{x_*}^2 \leq f(y_{t})-f(x_*)$.
	    Lastly, the bound on $\DD{x_t}{y_t}{z_t}$ follows from 
	  $$\DD{x_t}{y_{t}}{z_{t}} \leq \DD{x_t}{y_{t}}{x_*}+\DD{x_t}{z_{t}}{x_*}\leq  \dd{y_{t}}{x_*}+\DD{x_t}{z_{t}}{x_*}\, $$
	  which is a consequence of the (Euclidean) triangle inequality together with the fact that the projected distances are shorter than the actual distances (a property of Hadamard manifolds; see  e.g. \citep[Section 6.5]{burago2001course}).\end{proof}
From Proposition~\ref{prop:shrink}, we established that the  projected distance $\DD{x_t}{y_t}{z_t}$ is shrinking over iterations.
Now one can also demonstrate that $\dd{y_t}{z_t}$ is shrinking under mild conditions, and that is what the next proposition is about.
\begin{proposition} \label{dyz}
  Let $\init :=  f(x_0)-f(x_*) + \frac{1}{4\drop} \cc_{0}^2\cdot \dd{x_0}{x_*}^2$. If $\sg L >1$, $\sg L \leq 2-\cc_{t+1}$ and $\cc_{t+1}>2\mu\drop$ hold for $t\geq 0$, then Algorithm~\ref{alg:1} satisfies:
  $$\dd{y_t}{z_t} \leq \frac{\sqrt{\init\prod_{j=1}^{t}(1- \cc_j)}}{1-2\mu\drop \cc_{t+1}^{-1}}  \cdot \frac{ \left(\sqrt{\frac{2}{\mu}}+\sqrt{\frac{1}{\mu^2\drop}} + \frac{L}{\mu} \sqrt{\frac{2}{\mu}}  \right)(1-2\mu\drop)}{(\sg L -1) (\sg L-1+2\mu\drop)}\,. $$ 
\end{proposition}
\begin{remark} 
A careful reader might realize  that the appearance of the term  $1-2\mu\drop \cc_{t+1}^{-1}$  in the denominator of the bound could be potentially problematic as this could be arbitrarily small in general when $\cc_{t+1}$ is very close to $2\mu\drop$.
However, the appearance of this term will not be problematic for our purpose since the term will be canceled out with the algorithm parameter $\wm_{t+1} = 1-2\mu\drop \cc_{t+1}^{-1}$  (see Algorithm~\ref{alg:1}) when we use Proposition~\ref{dyz} to bound $\dd{x_t}{z_t}$.
\end{remark}
\begin{proof} 
First, from \eqref{r:nest:1} and \eqref{r:nest:2} together with the triangle inequality,
	\begin{align*}
	    \DD{x_{t+1}}{y_{t+1}}{z_{t+1}} &=\norm{-\sg \nabla f(x_{t
	    +1})-\wm_{t+1} \lm{x_{t+1}}{z_t} + \sm_{t+1}\nabla f(x_{t+1}) }_{x_{t+1}}\\
	  &\geq \wm_{t+1}\norm{\lm{x_{t+1}}{z_t}}_{x_{t+1}}-|\sm_{t+1}-\sg |\cdot \norm{\nabla f(x_{t+1}) }_{x_{t+1}}\\
    &=\wm_{t+1}\cdot \dd{x_{t+1}}{z_t}-|\sm_{t+1}-\sg |\cdot \norm{\nabla f(x_{t+1}) }\,.
	\end{align*}
	From \eqref{r:nest:0}, we have $\dd{x_{t+1}}{z_t} = (1-\wg_{t+1})\cdot \dd{y_t}{z_t}$. Hence, the above inequality becomes
	\begin{align*}
	    \wm_{t+1}(1-\wg_{t+1})\cdot \dd{y_t}{z_t} &\leq \DD{x_{t+1}}{y_{t+1}}{z_{t+1}} +|\sm_{t+1}-\sg|\cdot \norm{\nabla f(x_{t+1})} \\
	    &\leq \DD{x_{t+1}}{y_{t+1}}{z_{t+1}} +L|\sm_{t+1}-\sg|\cdot \dd{x_{t+1}}{x_*}\,,
	\end{align*}
	where the last inequality follows from the geodesically $L$-smoothness of $f$, which implies
	$\norm{\nabla f(x_{t+1})} \leq L\cdot \dd{x_{t+1}}{x_*}$.
Due to the Riemannian triangle inequality, we have $\dd{x_{t+1}}{x_*}\leq \dd{x_{t+1}}{y_t} +\dd{y_t}{x_*}$, and hence the upper bound becomes
	\begin{align*}
	& \DD{x_{t+1}}{y_{t+1}}{z_{t+1}} +L|\sm_{t+1}-\sg|\cdot  \dd{x_{t+1}}{y_t}+L|\sm_{t+1}-\sg|\cdot \dd{y_t}{x_*}, \\
	=&\DD{x_{t+1}}{y_{t+1}}{z_{t+1}} +L\wg_{t+1}|\sm_{t+1}-\sg|\cdot  \dd{y_t}{z_t}+L|\sm_{t+1}-\sg|\cdot \dd{y_t}{x_*}\,,
	\end{align*}
where the second line follows from the identity $\dd{x_{t+1}}{y_t} =\wg_{t+1} \cdot \dd{y_t}{z_t}$ ($\because$ \eqref{r:nest:0}).	
	Moving   $L\wg_{t+1}|\sm_{t+1}-\sg| \cdot  \dd{y_t}{z_t}$ term to the LHS, we then obtain:
	\begin{align} \label{bd:int}
	    \mathcal{K}\cdot \dd{y_t}{z_t} &\leq \DD{x_{t+1}}{y_{t+1}}{z_{t+1}} +L|\sm_{t+1}-\sg|\cdot \dd{y_t}{x_*}\,,
	\end{align}
where $\mathcal{K}:= \wm_{t+1}(1-\wg_{t+1}) - L\wg_{t+1} \left| \sm_{t+1}-\sg\right|$.

Having established \eqref{bd:int}, it is straightforward to see that Proposition~\ref{dyz} is a direct consequence the following two statements:
\begin{enumerate}
    \item The RHS is upper bounded by $\sqrt{\init\prod_{j=1}^{t}(1- \cc_j)} \cdot \left(\sqrt{\frac{2}{\mu}}+\sqrt{\frac{1}{\mu^2\drop}} + \frac{L}{\mu} \sqrt{\frac{2}{\mu}}  \right)$.
    \item $\mathcal{K}\geq  \frac{1-2\mu\drop \cc_{t+1}^{-1}}{1-2\mu\drop}\cdot  (\sg L -1) (\sg L-1+2\mu\drop)$, where the lower bound is always positive since $\sg L >1$ together with the fact that $1-2\mu\drop\cc_{t+1}^{-1} > 1-2\mu\drop \cdot (2\mu\drop)^{-1}=0$.
\end{enumerate}
Hence, the proof is completed as soon as we prove the above two statements.

We first prove the first statement. From Proposition~\ref{prop:shrink}, we have
\begin{itemize}
    \item $\DD{x_{t+1}}{y_{t+1}}{z_{t+1}} \leq  \sqrt{ \init \prod_{j=1}^{t+1}(1- \cc_j)} \cdot \left(\sqrt{\frac{2}{\mu}}+\sqrt{\frac{1}{\mu^2\drop}}\right)$
    \item $L|\sm_{t+1}-\sg|\cdot \dd{y_{t}}{x_*} \leq L|\sm_{t+1}-\sg|\cdot \sqrt{\init \prod_{j=1}^{t}(1- \cc_j)}\cdot \sqrt{\frac{2}{\mu}} $.
\end{itemize}
Moreover, since $\sm_{t+1} = 2\drop \cc_{t+1}^{-1} = \sg (2-\sg L)\cc_{t+1}^{-1}$, it is straightforward from  the assumption $\sg L \leq 2-\cc_{t+1}$ that $\sm_{t+1} \geq \sg $.
Hence, using inequalities (i) $1-\cc_{t+1}\leq 1$ and (ii) $L|\sm_{t+1}-\sg|\leq L\sm_{t+1}= 2L\drop \cc_{t+1}^{-1}  \leq \frac{L}{\mu}$, we obtain the first statement.

Now, let us prove the second statement. 
From the parameter choices in Algorithm~\ref{alg:1} together with the fact $\sm_{t+1}\geq \sg$ we have established above, one can simplify and lower bound $\mathcal{K}$ as follows:
	    \begin{align*}
	  \mathcal{K} &=(1-2\mu\drop \cc_{t+1}^{-1})\frac{1-\cc_{t+1}}{1-2\mu\drop} - L  \frac{\cc_{t+1}-2\mu\drop }{1-2\mu\drop}\left(2\drop\cc_{t+1}^{-1}-\sg\right) \\
	        &=  \frac{1-2\mu\drop \cc_{t+1}^{-1}}{1-2\mu\drop}\cdot \left[ 1-\cc_{t+1} - 2L\drop  +\sg L \cc_{t+1} \right]\\
 &=\frac{1-2\mu\drop \cc_{t+1}^{-1}}{1-2\mu\drop}\cdot \left[ (\sg L -1)^2 +(\sg L-1)\cc_{t+1} \right]\\
 &>\frac{1-2\mu\drop \cc_{t+1}^{-1}}{1-2\mu\drop}\cdot \left[ (\sg L -1)^2 +(\sg L-1)\cdot 2\mu\drop \right]\,.   \end{align*}
 where the last line follows from the fact $\cc_{t+1}>2\mu\drop$ and $\sg L -1 >0$.
 \end{proof}
Now we are ready to provide the formal statement and the proof of Lemma~\ref{lem:shrink:d}:
  \begin{lembox}[Formal statement of Lemma~\ref{lem:shrink:d}] \label{lem:shrink:f}
	   Assume that $\mu>0$.
	   Let   $\init :=  f(x_0)-f(x_*) + \frac{1}{4\drop} \cc_{0}^2\cdot \dd{x_0}{x_*}^2$.
	  If $\sg L >1$, $\sg L \leq 2-\cc_{t+1}$ and $\cc_{t+1}>2\mu\drop$  hold, then  Algorithm~\ref{alg:1} satisfies:
	   $$\dd{x_{t+1}}{z_{t+1}} \leq   \co \cdot  \sqrt{\init\prod_{j=1}^{t}(1- \cc_j)}\,,$$
	   where  $\co = \frac{ \left(\sqrt{\frac{2}{\mu}}+\sqrt{\frac{1}{\mu^2\drop}} + \frac{L}{\mu} \sqrt{\frac{2}{\mu}}  \right)(2L\drop+ 1-2\mu\drop)}{(\sg L -1) (\sg L-1+2\mu\drop)} + \frac{L}{\mu}\sqrt{\frac{2}{\mu}}$.
	     \end{lembox}
\begin{proof}
  From \eqref{r:nest:2}, one can use the Euclidean triangle inequality on $\T_{x_{t+1}}M$ to obtain: 
\begin{align*}
\dd{x_{t+1}}{z_{t+1}}& \leq \wm_{t+1} \cdot \dd{x_{t+1}}{z_t} + \sm_{t+1} \cdot \norm{\nabla f(x_{t+1})}_{x_{t+1}}\\
&\overset{(\clubsuit)}{\leq}  \wm_{t+1} \cdot \dd{x_{t+1}}{z_t} + L\sm_{t+1} \cdot  \dd{x_{t+1}}{x_*}\\
&\overset{(\spadesuit)}{\leq} \wm_{t+1} \cdot \dd{x_{t+1}}{z_t} + L\sm_{t+1} \cdot  \dd{x_{t+1}}{y_t}+L\sm_{t+1} \cdot  \dd{y_t}{x_*}\\
&\overset{(\heartsuit)}{=}\left( \wm_{t+1}(1-\wg_{t+1}) + L\sm_{t+1} \wg_{t+1}  \right) \cdot \dd{y_t}{z_t} +L\sm_{t+1} \cdot  \dd{y_t}{x_*}\\
&\overset{(\diamondsuit)}{=}
\frac{2L\drop+ 1-\cc_{t+1}}{1-2\mu\drop}(1-2\mu\drop\cc_{t+1}^{-1})\cdot \dd{y_t}{z_t} +2L\drop \cc_{t+1}^{-1}  \cdot  \dd{y_t}{x_*}\,,
\end{align*}
where ({\footnotesize $\clubsuit$}) is due to the geodesically $L$-smoothness of $f$, which implies   
$\norm{\nabla f(x_{t+1})} \leq L\cdot \dd{x_{t+1}}{x_*}$; ({\footnotesize $\spadesuit$}) is due to Riemannian triangle inequality; ({\footnotesize $\heartsuit$}) is due to \eqref{r:nest:0}; and
  ({\footnotesize $\diamondsuit$}) follows from  from the choice of parameters in Algorithm~\ref{alg:1}.

Now  after we apply Propositions~\ref{prop:shrink} and \ref{dyz} to the last upper bound, and use the fact $\cc_{t+1}\in [2\mu\drop,1)$ to upper bound $\cc_{t+1}$'s in the resulting upper bound, Lemma~\ref{lem:shrink:f} readily follows.
	\end{proof}

	\subsection{Proof of Theorem~\ref{thm:main}}
	\label{pf:main}
	We first  demonstrate that regardless of what initial value $\cc_0\geq 0$ we choose, $\cc_t$ becomes less than $\sqrt{\mu/L}$  after few iterations.
	Before the demonstration, we denote by $\cc_{t+1} = \ctt_{t+1}(\cc_t)$ the relation with which $\{\cc_{t}\}$ is produced in Algorithm~\ref{alg:1}.
	In other words, $\cc_{t+1} = \ctt_{t+1}(\cc_t)$ is equivalent to \eqref{r:recur:xi} with $\ddr{} =\ddr{t+1}$.  
	\begin{proposition} \label{xiconv}
	If $\cc_0 > \sqrt{ \mu/L}$, then $\cc_t\leq \sqrt{\mu/L}$ for all $t$ greater than or equal to 
	\begin{align} \label{num:iter}
	   \frac{ \log\left((\cc_0-\sqrt{2\mu\drop})/(\sqrt{\mu/L}-\sqrt{2\mu\drop})\right)}{\log\left(1/\left(1-\frac{8\mu\drop}{5+\sqrt{5}}\right)\right)}\,.
	\end{align}
	If $\cc_0<\sqrt{\mu/L}$, then $\cc_t\leq \sqrt{\mu/L}$ for all $t\geq0$.
	\end{proposition}
	\begin{proof}
At some iteration $t$, we consider the two cases depending on whether $\cc_{t}\leq \sqrt{2\mu\drop}$ or not:
\begin{enumerate}
    \item First, if $\cc_t\leq \sqrt{2\mu\drop}$, then we have $\cc_{t'}\leq  \sqrt{2\mu\drop}$ for all $t'\geq t$ because $\cc(\ddr{t'})$ is less than $\sqrt{2\mu\drop }$ for all $t'$ and the  relation at each step, $\cc_{t'+1} = \ctt_{t'+1}(\cc_{t'})$, will only bring $\cc_t'$ closer to $\cc(\ddr{t'})$ due to Lemma~\ref{lem:appen}.
    \item Now consider the case  $\cc_{t}>\sqrt{2\mu\drop}$. We may assume that $\cc_{t+1} >\sqrt{2\mu\drop}$ (otherwise, $\cc_{t'}\leq \sqrt{2\mu\drop}$ for $t'\geq t+1$ due to the first case).
    Then, the mean value theorem implies:
    \begin{align*}
	    \cc_{t+1} -\sqrt{2\mu\drop }&=\ctt_{t+1}(\cc_{t}) -\ctt_{t+1}(\ctt_{t+1}^{-1}(\sqrt{2\mu\drop }))\\
	    &\overset{(\clubsuit)}{\leq}    \frac{1}{\sqrt{\ddr{t+1}}} \left(1-\frac{4}{5+\sqrt{5}}\cdot \frac{2\mu\drop}{\sqrt{\ddr{t+1}}} \right)\cdot \left(\cc_t-\ctt_{t+1}^{-1}(\sqrt{2\mu\drop })\right)\\
	    &\overset{(\spadesuit)}{<}\left(1-\frac{4}{5+\sqrt{5}}\cdot2\mu\drop \right)\cdot \left(\cc_t-\sqrt{2\mu\drop}\right)\,, 
	\end{align*}  
	where  ({\footnotesize $\clubsuit$})  is due to Proposition~\ref{prop:der} together with $\cc_{t+1}>\sqrt{2\mu\drop} \Rightarrow \cc_t > \ctt_{t+1}^{-1}(\sqrt{2\mu\drop})$;
	({\footnotesize $\spadesuit$}) follows since  the maximum of $\frac 
	{1}{\sqrt{\ddr{}}}(1-\frac{4}{5+\sqrt{5}}\frac{2\mu\drop}{\sqrt{\ddr{}}})$ is achieved by $\ddr{}=1$ and  $\sqrt{2\mu\drop }<\ctt_{t+1}^{-1}(\sqrt{2\mu\drop }) $ due to  $\sqrt{2\mu\drop} \geq \cc(\ddr{t+1})$ and Lemma~\ref{lem:appen}.
	Hence, the distance between $\cc_t$ and $\sqrt{2\mu\drop}$ shrinks geometrically.
\end{enumerate}
Combining the two cases, we conclude the proof.
	\end{proof}
We now study the rate of convergence of $\{\cc_t\}$.
To that end, we first study the convergence of $\{\cc(\ddr{t})\}$.
For simplicity, we assume that $\cc_0\leq \sqrt{\mu/L}$.
By Proposition~\ref{xiconv}, the arguments below remain true for $\cc_0>\sqrt{\mu/L}$ after we substitute $t\leftarrow t+ \eqref{num:iter}$.
We first characterize the behaviour of $\cc(\ddr{})$ near $\ddr{}=1$:
\begin{proposition}\label{behav:1}
  Let $\cc(\ddr{}):=\frac{1}{2}\left(\sqrt{(\ddr{}-1)^2+8\ddr{}\mu\drop }-(\ddr{}-1)\right)$ for $\ddr{}\geq 1$.
  Then, $0\leq \sqrt{2\mu\drop}- \cc(\ddr{}) \leq \frac{1}{2}(\ddr{}-1)$ for $1\leq \ddr{} \leq 1+3/(1+ (4\mu\drop)^{-1})$.
\end{proposition}
\begin{proof}
For simplicity, let us write $\ddr{}=1+d$.
Then, $\cc(1+d) =\frac{1}{2}\left(\sqrt{d^2 + 8\mu\drop (1+d)}-d\right)$.
Using the inequality $\sqrt{1+r}\geq 1+\frac{1}{3}r$ for $0\leq r\leq 3$, we get the following as long as $d+\frac{1}{8\mu\drop}d^2 \leq 3$:
\begin{align*}
    \cc(1+d) &\geq   \sqrt{2\mu\drop}\cdot \left(1+ \frac{1}{3}d + \frac{1}{24\mu\drop}d^2\right) -\frac{1}{2}d\\
    &\geq  \sqrt{2\mu\drop} -\left(\frac{1}{2}-\frac{\sqrt{2\mu\drop}}{3} \right) d\,.
\end{align*}
Now all we need to check is that $d \leq 3/(1+\frac{1}{4\mu\drop})$ implies $d+\frac{1}{8\mu\drop}d^2 \leq 3$.
Indeed, if $d \leq 3/(1+\frac{1}{4\mu\drop})$, then we have $d \leq 3/(1+\frac{1}{4\mu\drop}) \leq 3/ (3/2)=2$, and hence  $d+\frac{1}{8\mu\drop}d^2 = d\left(1+\frac{d}{8\mu\drop}\right) \leq d\left(1+\frac{1}{4\mu\drop}\right)\leq 3$.
\end{proof}
Next, we characterize the behaviour of the function $\vd(r)$ near $r=1$.
\begin{proposition}\label{behav:2}
  $\vd(r) \leq 1+ 2\kappa r^2$ for $0\leq r \leq \frac{1}{2\sqrt{\kappa}}$ .
\end{proposition}
\begin{proof}
Using Taylor expansion, one easily easily verify  for $0\leq r \leq \frac{1}{2\sqrt{\kappa}}$ that
\begin{align*}
    \frac{\sqrt{\kappa}r}{\tanh (\sqrt{\kappa}r)}\leq 1+\frac{\kappa}{2} r^2\quad\text{and} \quad \left(\frac{\sinh(2\sqrt{\kappa}r)}{2\sqrt{\kappa}r}\right)^2 \leq 1+ 2\kappa r^2\,.
\end{align*}
Hence, from the definition of $\vd$ (see \eqref{def:vd}), we obtain the desired bound on $\vd$.
\end{proof}
Combining Propositions~\ref{behav:1} and \ref{behav:2}, we obtain the following results:
\begin{proposition} \label{behav}
$\sqrt{2\mu\drop}-\cc(\vd(r) ) \leq \kappa r^2$   for $0\leq r\leq \sqrt{\frac{3}{1+(4\mu\drop)^{-1}}}\cdot  \frac{1}{2\sqrt{\kappa}}$.
\end{proposition}
\begin{proof}
Note that $\frac{3}{1+(4\mu\drop)^{-1}} \leq \frac{3}{1+2L/\mu} \leq 1$, and hence, $\sqrt{\frac{3}{1+(4\mu\drop)^{-1}}}\cdot  \frac{1}{2\sqrt{\kappa}} \leq \frac{1}{2\sqrt{\kappa}}$.
Thus, one can apply Proposition~\ref{behav:2} for $0\leq r\leq \sqrt{\frac{3}{1+(4\mu\drop)^{-1}}}\cdot  \frac{1}{2\sqrt{\kappa}}$, and obtain $\vd(r) \leq 1+2\kappa r^2$.
Hence, $\vd(r) \leq 1+ \frac{1}{2}\cdot \frac{3}{1+(4\mu\drop)^{-1}}$ within the range.
Hence, by Proposition~\ref{behav:1}, one then obtains
$\sqrt{2\mu\drop}- \cc(\vd(r))  \leq \kappa r^2$ for  $0\leq r\leq \sqrt{\frac{3}{1+(4\mu\drop)^{-1}}}\cdot  \frac{1}{2\sqrt{\kappa}}$. 
\end{proof}

Now let $\ct:= \sqrt{\frac{3}{1+(4\mu\drop)^{-1}}}\cdot  \frac{1}{2\sqrt{\kappa}}$.
Then by Lemma~\ref{lem:shrink:f}, one can deduce that  $\dd{x_{t+1}}{z_{t+1}} \leq \ct$ whenever $t \geq  2 \frac{\log(\co \cdot \sqrt{D_0}/ \ct)}{-\log (1-2\mu\drop)}$.
Therefore,   Proposition~\ref{behav} implies that for $t \geq  2 \frac{\log(\co \cdot \sqrt{D_0}/ \ct)}{-\log (1-2\mu\drop)}$, the following bound holds:
\begin{align*}
    \sqrt{2\mu\drop}- \cc\big(\vd\big(\dd{x_{t+1}}{z_{t+1}}\big)\big) \leq \kappa \co^2 D_0 (1-2\mu\drop)^{t}\,.
\end{align*}
From this, it follows that $\cc\big(\vd\big(\dd{x_{t+1}}{z_{t+1}}\big)\big)  \in [\sqrt{2\mu\drop}-\eps/2, \sqrt{2\mu\drop}]$ whenever 
\begin{align*}
    t \geq \max \left\{  2 \frac{\log(\co \cdot \sqrt{D_0}/ \ct)}{-\log (1-2\mu\drop)},~ \frac{\log(2\kappa \co^2D_0/\eps)}{-\log (1-2\mu\drop)} \right\}\,.
\end{align*}

Now having established   the convergence of $\{\cc(\ddr{t})\}$, we translate it into the convergence of $\{\cc_t\}$.
Similarly to the proof of Proposition~\ref{xiconv}, one can prove that for any $T\geq 0$,
\begin{align*}
    |\cc_{T+t} -\cc(\ddr{T})| \leq \left(1-\frac{8\mu\drop }{5+\sqrt{5}}\right)^t 
    |\cc_{T}-\cc(\ddr{T})|\,.
\end{align*}
From this, one can conclude that $\cc_{t+1} \in [\sqrt{2\mu\drop}-\eps, \sqrt{2\mu\drop}]$ whenever
\begin{align*}
    t \geq \max \left\{  2 \frac{\log(\co \cdot \sqrt{D_0}/ \ct)}{-\log (1-2\mu\drop)},~ \frac{\log(2\kappa \co^2D_0/\eps)}{-\log (1-2\mu\drop)} \right\} +\frac{\log (2\sqrt{2\mu\drop}/\eps)}{-\log\left(1-\frac{8\mu\drop }{5+\sqrt{5}} \right)}\,,
\end{align*}
concluding the proof of the the convergence rate of $\{\cc_t\} $ in Theorem~\ref{thm:main}.

	\subsection{Justification of Remark~\ref{rmk:constant}} \label{appen:just}
	In this section, we verify that $\cc(\ddr{})$ is decreasing.
Note that for $\ddr{}\geq 1$ we have
		\begin{align*}
		\frac{d}{d\ddr{}}\cc(\ddr{})=\frac{2(\ddr{}-1)+8\mu\drop}{4\sqrt{(\ddr{}-1)^2+8\mu\drop\ddr{} }}-\frac{1}{2} =  \frac{2(\ddr{}-1)+8\mu\drop  -2\sqrt{(\ddr{}-1)^2+8\mu\drop\ddr{} }}{4\sqrt{(\ddr{}-1)^2+8\mu\drop\ddr{}}}< 0\,,
		\end{align*}
		where the last inequality is due to the fact that $\left((\ddr{}-1)+4\mu\drop \right)^2 = (\ddr{}-1)^2 +8\mu\drop (\delta-1) + 16\mu^2\drop^2 <(\ddr{}-1)^2 +8\mu\drop\delta-8\mu\drop(1-2\mu\drop) < (\ddr{}-1)^2 +8\mu\drop\delta$ since $2\mu\drop <1$.

\section{Extension to the non-Hadamard case}
\label{sec:nonhada}

Let us now assume that  the sectional curvatures of $M$ is upper  bounded by $\upp\geq 0$.
In particular, $\upp=0$ corresponds to the Hadamard case.
We first pinpoint the main differences:
Unlike the Hadamard case, $M$ now may not be uniquely geodesic.
Instead, one can only guarantee the property within a local neighborhood of $M$.
Consequently, the notion of convexity can be only guaranteed within a local neighborhood of $M$.
For instance, manifolds with positive sectional curvatures (e.g. spheres) are compact, and hence, they do not admit globally geodesically convex functions other than the constant function.  
Following the prior arts~\citep{dyer2015riemannian,zhang2018estimate}, we make the following assumptions to avoid any further complications:
\begin{assump}
The domain $N\subset M$ of $f$ is uniquely geodesic with the diameter bounded by $\frac{\pi}{2\sqrt{\upp}}$. 
\end{assump}
\begin{assump}
All the iterates of \eqref{r:nesterov} (whose parameters to be chosen later) remain in $N$.
\end{assump}

The analysis for the non-Hadamard case is identical to that for the Hadamard case, modulo one additional geometric inequality due to \citep{zhang2018estimate}:
\begin{proposition}[{\citep[Lemma 7]{zhang2018estimate}}] \label{prop:zhang}
	Let $x,y,z$ be points on Riemannian manifold $M$  with  sectional curvatures  upper bounded by  $\upp\geq0$. If $d(x,z) \leq \frac{\pi}{2\upp}$, then 
	\begin{align*}
	    \DD{x}{y}{z}^2 \leq (1+2\cdot \dd{x}{y}^2)\cdot \dd{y}{z}^2\,.
	\end{align*}
\end{proposition}
Applying Proposition~\ref{prop:zhang} to Lemma~\ref{dist:formal}, we obtain the following metric distortion inequality:
\begin{lembox}[Modification of Lemma~\ref{dist:formal}]  \label{dist:formal:n}
		Let $x,x',y,z$ be points on Riemannian manifold $M$ with sectional curvatures upper and lower bounded by $\upp$ and $-\kappa<0$, respectively.
		For $\widehat{\vd}:\re_{\geq 0}\to \re_{\geq 1}$ defined as in Lemma~\ref{dist:formal}, we have
		\begin{align*}
		    \DD{x'}{y}{z}^2\leq \widehat{\vd}(\dd{x}{y})\cdot (1+2\cdot \dd{x'}{y}^2)\cdot \DD{x}{y}{z}^2\,.
		\end{align*}
\end{lembox}
 From Lemma~\ref{dist:formal:n}, one can conclude that at iteration $t$,
 \begin{align}\label{vd:nonhad}
     \vd{\dd{x_t}{z_t}} \cdot (1+2\cdot \dd{y_t}{z_t}^2)   
 \end{align}
is a valid distortion rate.
Thus, one can use \eqref{vd:nonhad} in lieu of $\vd(\dd{x_t}{z_t})$ for the valid distortion rate in Algorithm~\ref{alg:1}.
Then the rest proceeds in the exactly same manner; in particular, one can similarly establish the distance shrinking results as in Appendix~\ref{pf:shrink} to corroborate Theorem~\ref{thm:main} for this case.
We skip the details since they significantly overlap with the Hadamard case.

\end{document}